\documentclass[12pt,a4paper]{amsart}
\usepackage{graphics}
\usepackage{amscd}
\usepackage{color}
\usepackage{comment}
\usepackage{oldgerm}
\usepackage{amsmath}
\usepackage{amssymb}
\usepackage{amsthm}
\usepackage{bm}
\newtheorem{thm}{Theorem}[section]
\newtheorem{prop}[thm]{Proposition}
\newtheorem{lem}[thm]{Lemma}
\newtheorem{cor}[thm]{Corollary}
\newtheorem{definition}[thm]{Definition}
\newtheorem{rem}[thm]{Remark}

\newcommand{\Z}{\mathbb{Z}}

\newcommand{\R}{\mathbb{R}}
\newcommand{\C}{\mathbb{C}}
\newcommand{\T}{\mathbb{T}}
\newcommand{\V}{\mathbb{V}}

\newcommand{\vol}{{\rm vol}}

\newcommand{\ddb}{\partial\overline{\partial}}
\newcommand{\del}{\partial}

\title[The harmonic $2$ forms on $K3$ surfaces]{
The harmonic $2$-forms on $K3$ surfaces converging to a flat $4$-dimensional orbifold}
\author[K. Hattori]{Kota Hattori}
\address{Keio University, 
3-14-1 Hiyoshi, Kohoku, Yokohama 223-8522, Japan}
\email{hattori@math.keio.ac.jp}

\begin{document}
\maketitle

\begin{abstract}
In this article, we study the asymptotic behavior of harmonic $2$-forms on $K3$ surfaces with Ricci-flat K\"ahler metrics, where metrics converge to the quotient of a flat $4$-dimensional torus by a finite group action. We can show that the space of anti-self-dual harmonic $2$-forms decomposes into two subspaces: one converges to the flat $2$-forms on the quotient of the torus, while the other converges to the first Chern forms of anti-self-dual connections on ALE spaces. 
\end{abstract}

\section{Introduction}
The study of Ricci-flat K\"ahler metrics on Calabi-Yau manifolds is one of the significant objects in geometric analysis on complex manifolds. On a compact K\"ahler manifold $(X,\omega)$, the Chern class of the canonical line bundle is represented by the Ricci-form of $\omega$. If the canonical bundle is trivial, then the Chern class vanishes. In this case, there is a unique Ricci-flat K\"ahler metric in every given K\"ahler class by Yau's Theorem \cite{Yau1978}. In general, the Ricci-flat K\"ahler metrics are obtained by solving Monge-Amp\`ere equation and we cannot understand their geometric aspects explicitly. However, there are several methods to construct Ricci-flat K\"ahler metrics gluing some singular Ricci-flat K\"ahler metrics with some local models such as gravitational instantons. 

The typical examples of such metrics can be constructed on Kummer surfaces. Let $\T^4=\C^2/\Z^4$ be a $4$-torus and we have the natural action of $\Z_2=\{ \pm 1\}$. Then the quotient $\T^4/\Z_2$ has $16$ orbifold singularities modelled by $\C^2/\Z_2$, whose minimal resolution is biholomorphic to the holomorphic cotangent bundle $T^*\C P^1$ of the complex projective line. Here, the Euclidean metric on $\C^2$ descends to the flat orbifold metric on $\T^4/\Z_2$ and $\C^2/\Z_2$, respectively, and $T^*\C P^1$ has the complete Ricci-flat K\"ahler metrics called the Eguchi-Hanson metrics asymptotic $\C^2/\Z_2$ with the flat metric. Gluing the flat K\"ahler metrics on $\T^4/\Z_2$ with the $16$ copies of Eguchi-Hanson metrics, we obtain the Ricci-flat K\"ahler metrics on $K3$ surfaces. The idea of this geometric picture has been suggested by Page \cite{Page1978} and Gibbons and Pope \cite{GibbonsPope1979}, then Kobayashi and Todorov interpreted it from the viewpoint of the period map of marked polarized $K3$ surfaces \cite{KobayashiTodorov}. A rigorous argument for the gluing procedure was shown by Arrezo and Pacard \cite{ArezzoPacard}, Donaldson \cite{Donaldson2010}, and Jiang \cite{jiang2025kummer}. 

Another context of this research comes from the spectral convergence of the Hodge Laplacian on Riemannian manifolds. For a compact Riemannian manifold $(X,g)$, the Hodge Laplacian is defined by $\Delta_g^k=dd^*+d^*d\colon \Omega^k(X)\to \Omega^k(X)$, where $d^*$ is the formal adjoint of the exterior derivative $d$. By the harmonic analysis, there are nonnegative eigenvalues $0\le \lambda^k_1(g)\le \lambda^k_2(g)\le\cdots\to \infty$ and a complete orthonormal system $(\varphi^k_j)_{j=1}^\infty$ of $L^2(X,\Lambda^kT^*X)$ such that $\Delta^k_g\varphi^k_j=\lambda^k_j(g)\varphi^k_j$. In the case of $k=0$, $\Delta_g^0$ is the Laplace-Beltrami operator acting on $C^\infty(X)$, and the theory of spectral convergence has been studied well. Here, the spectral convergence means the convergence of the $j$-th eigenvalues $(\lambda^0_j(g_i))_{i=0}^\infty$ and corresponding eigenfunctions to the ones defined on the limit space under the measured Gromov-Hausdorff convergence of $(X_i,g_i)\stackrel{i\to \infty}{\to} (X_\infty,d,m)$. In general, the limit space may not be smooth, hence we consider metric measure spaces. If the diameters of $(X_i,g_i)$ and the sectional curvatures $|{\rm sec}_{g_i}|$ are uniformly bounded, then the spectral convergence of $\Delta^0_{g_i}$ was shown by Fukaya \cite{Fukaya1987}, then shown by Cheeger and Colding \cite{Cheeger-Colding3} under the assumption of the bounded diameters and giving the uniform lower bound of the Ricci curvatures. In the case of $k>0$, Honda showed the spectral convergence of $\Delta_{g_i}^1$ in \cite{Honda2017spectral} under the measured Gromov-Hausdorff convergence, where the diameters are bounded and the Ricci curvatures are two-sided bounded. His result can be applied to the sequence of Ricci-flat K\"ahler metrics $g_i$ on a Kummer surface $X$, where all of the areas of $16$ exceptional curves converge to $0$. 

However, the examples of a sequence of Ricci-flat K\"ahler metrics on Kummer surfaces tell us the spectral convergence of $\Delta^2_{g_i}$ never holds. The $2$nd betti number of the $K3$ surface is known to be $22$, consequently, $\lambda^2_j(g_i)=0$ for all $1\le j\le 22$ and $\lambda^2_j(g_i)>0$ for all $j>22$. On the other hand, if $g_0$ is the flat metric on $\T^4/\Z_2$, the $0$-eigenspace of $\Delta^2_{g_0}$ corresponds to the space of $\Z_2$-invariant flat $2$-forms on $(\T^4,g_0)$, whose dimension is equal to $6$. Therefore, the sequence $(\lambda^2_j(g_i))_{i=0}^\infty$ never converge to the $j$-th eigenvalue of $\Delta^2_{g_0}$ for $6<j\le 22$. The main purpose of this article is to observe the behavior of the harmonic $2$-forms of $g_i$ as $(X,g_i)$ converging to the flat orbifold $\T^4/\Z_2$.

In this article, we consider the following setting. Let $\T^4_\Lambda=\C^2/\Lambda$ be a flat torus, where $\Lambda\subset \C^2$ is a lattice of maximal rank. Let $\Gamma$ be a nontrivial finite subgroup of $SU(2)$, which acts on $\T^4_\Lambda$ isometrically and preserving the complex structure, then we obtain a flat K\"ahler orbifold $\T:=\T^4_\Lambda/\Gamma$. Denote by $S$ the subset of orbifold singularities, then for every $p\in S$ there is a subgroup $\Gamma_p\subset\Gamma$ such that the neighborhood of $p$ in $\T$ is isomorphic to the neighborhood of the origin in $\C^2/\Gamma_p$. For example, if $\Gamma=\Z_2$, then $\# S=16$ and $\Gamma_p\cong\Z_2$ for all $p\in S$. Next, we take a minimal resolution around every $p\in S$, then obtain a smooth compact complex manifold  $X$ and a map $\pi\colon X\to \T$ such that $\pi|_{X\setminus \pi^{-1}(S)}\colon X\setminus \pi^{-1}(S)\to \T\setminus S$ is biholomorphic. For the standard holomorphic coordinate $(z,w)$ on $\C^2$, the holomorphic $2$-form $dz\wedge dw$ descends to $\T\setminus S$, then it is known that $\pi^*(dz\wedge dw)$ extends to the nowhere vanishing holomorphic $2$ form $\Omega$ on $X$. Moreover, we can also take a minimal resolution $\varpi_p\colon M_p\to\C^2/\Gamma_p$. Then, the underlying smooth manifold of $X$ can be obtained by 
\begin{align*}
X=(\T\setminus S)\cup \bigcup_{p\in S}M_p,
\end{align*}
where the intersection $(\T\setminus S)\cap M_p$ is diffieomorphic to $\R_+\times S^3/\Gamma_p$ and $M_p\cap M_q$ is empty if $p\neq q$.

On $M_p$, if $\alpha\in H^2(M_p)$ satisfies $\alpha(E)\neq 0$ for every $E\in H_2(M_p,\Z)$ with $E\cdot E=-2$, then Kronheimer showed that there is a unique Ricci-flat K\"ahler form $\eta_p\in \Omega^{1,1}(M_p)$ satisfying the ALE condition \cite{kronheimer1989construction}. Moreover, Kronheimer and Nakajima constructed vector bundles with anti-self-dual connections $(\mathcal{R}_i,A_i)$ on $M_p$ such that the first Chern classes $\{ c_1(\mathcal{R}_i)\}_{i=1,\ldots, N_{\Gamma_p}}$ form a basis of $H^2(M_p)$ \cite{KN1990instanton}. Here, $N_{\Gamma_p}$ is the number of the isomorphic classes of nontrivial irreducible representations of $\Gamma_p$. Here, $c_1(\mathcal{R}_i)$ is represented by the Chern form $c_1(A_i)$ which is an anti-self-dual $2$-form. 

For a K\"ahler metric $\omega$ on $X$, denote by $\mathcal{H}^2_\omega$ the space of harmonic $2$-forms on $X$. 
Moreover, we have the orthogonal decomposition $\mathcal{H}^2_\omega=\mathcal{H}^2_{\omega,+}\oplus \mathcal{H}^2_{\omega,-}$, where $\mathcal{H}^2_{\omega,\pm}=\{ \alpha\in\mathcal{H}^2_\omega|\, *\alpha=\pm \alpha\}$. If $\omega$ satisfies the Monge-Amp\`ere equation $\omega^2=\frac{1}{2}\Omega\wedge\bar{\Omega}$, then $\omega$ is Ricci-flat and $\mathcal{H}^2_{\omega,+}={\rm span}\{ \omega,{\rm Re}(\Omega),{\rm Im}(\Omega)\}$. Similarly, we have the decomposition of harmonic $2$-forms on $\T$ written as $\mathcal{H}^2_\T=\mathcal{H}^2_{\T,+}\oplus \mathcal{H}^2_{\T,-}$. Here, we regard $\mathcal{H}^2_{\T,\pm}$ as the space of $\Gamma$-invariant harmonic (anti-)self-dual $2$-forms on $\T^4_\Lambda$. Put $d_\Gamma:=\dim\mathcal{H}^2_{\T,-}$. 

\begin{thm}
There is a family of closed $2$-forms on $X$ 
\begin{align*}
\left( \tilde{\omega}_\varepsilon,{\rm Re}(\Omega),{\rm Im}(\Omega),(\tilde{\omega}^-_{\alpha,\varepsilon})_{\alpha=1,\ldots,d_\Gamma},(\tilde{\Xi}_{p,i,\varepsilon})_{p\in S,i=1,\ldots,N_{\Gamma_p}}\right)_{0<\varepsilon\le \varepsilon_0}
\end{align*}
satisfying the following properties. 
\begin{itemize}
\setlength{\parskip}{0cm}
\setlength{\itemsep}{0cm}
 \item[$({\rm i})$] The cohomology classes $[\tilde{\omega}_\varepsilon],[{\rm Re}(\Omega)],[{\rm Im}(\Omega)],[\tilde{\omega}^-_{\alpha,\varepsilon}],[\tilde{\Xi}_{p,i,\varepsilon}]$ form a basis of $H^2(X)$.
 \item[$({\rm ii})$] $\tilde{\omega}_\varepsilon$ satisfies the Monge-Amp\`ere equation $\tilde{\omega}_\varepsilon^2=\frac{1}{2}\Omega\wedge\bar{\Omega}$. In particular, $\tilde{\omega}_\varepsilon,{\rm Re}(\Omega),{\rm Im}(\Omega)$ form a basis of $\mathcal{H}^2_{\tilde{\omega}_\varepsilon,+}$.
 \item[$({\rm iii})$] $\tilde{\omega}^-_{\alpha,\varepsilon},\tilde{\Xi}_{p,i,\varepsilon}$ form a basis of $\mathcal{H}^2_{\tilde{\omega}_\varepsilon,-}$. 
 \item[$({\rm iv})$] There are $\omega_0 \in \mathcal{H}^2_{\T,+}$ and $\omega^-_1,\ldots,\omega^-_{d_\Gamma}\in \mathcal{H}^2_{\T,-}$ such that the following holds. As $\varepsilon\to 0$, we have the following convergence 
\begin{align*}
\left( \tilde{\omega}_\varepsilon,\tilde{\omega}^-_{\alpha,\varepsilon}, \tilde{\Xi}_{p,i,\varepsilon}\right)|_{\T\setminus S}
\to \left( \omega_0,\omega^-_\alpha, 0\right)
\end{align*}
in $C^k$-topology on every compact subset in $\T\setminus S$ and $\omega_0,{\rm Re}(dz\wedge dw),{\rm Im}(dz\wedge dw)$ form a basis of $\mathcal{H}^2_{\T,+}$ and $\omega^-_1,\ldots,\omega^-_{d_\Gamma}$ form a basis of $\mathcal{H}^2_{\T,-}$.
 \item[$({\rm v})$] Let $p\in S$. On the neighborhood of $\pi^{-1}(p)$, a Ricci-flat K\"ahler manifold $(M_p,\eta_p)$ bubbles off as a limit of $\varepsilon^{-2}\tilde{\omega}_\varepsilon$ as $\varepsilon\to 0$. Moreover, on the neighborhood of $\pi^{-1}(p)$, the first Chern form $c_1(A_{p,i})$ on $M_p$ bubbles off as a limit of $\tilde{\Xi}_{p,i,\varepsilon}$ as $\varepsilon\to 0$. 
\end{itemize}
\end{thm}
The meaning of the term ``bubble" in this article will be explained in Definition \ref{def conv bubble}. Roughly speaking, it refers to a refined description of geometric objects near singularities. Such bubbling phenomenon were known in several situation. For example, the construction of $\tilde{\omega}_\varepsilon$ and its asymptotic behavior has been already shown by \cite{KobayashiTodorov}\cite{ArezzoPacard}\cite{Donaldson2010}\cite{jiang2025kummer}. Recently, Tazoe described the second Chern form of the Levi-Civita connection of Ricci-flat K\"ahler metric around the singularities \cite{tazoe2025}. 

This paper is organized as follows. In Section \ref{sec Kummer}, we explain the basic setting of this paper and construct a compact complex surface $X$ and approximating K\"ahler metrics. In Section \ref{sec MA}, we construct a family of Ricci-flat K\"ahler metrics on $X$. The results in Sections \ref{sec Kummer} and \ref{sec MA} is not new, we follow the argument in \cite{jiang2025kummer}. Next, we construct a family of harmonic $2$-forms on $X$ with respect to the Ricci-flat K\"ahler metrics introduced in Section \ref{sec MA}. These harmonic $2$-forms decompose naturally into two families: One converges to the harmonic forms on the flat torus, and the other converges to those on ALE gravitational instantons. 
In Section \ref{sec ASD ALE}, we construct the approximating closed $2$ forms on $X$ converging to the harmonic forms on a flat torus, and in Section \ref{sec ASD torus}, we construct closed $2$-forms converging to the harmonic forms on ALE spaces. We obtain the precise statement of the main theorem in Section \ref{sec main}. We give some examples in Section \ref{sec ex}

\vspace{0.1cm}
{\bf Acknowledgement.}
The author would like to thank Shouhei Honda for suggesting the problems considered in this article. The author is also grateful to Itsuki Tazoe for helpful advice and valuable discussions. This work was supported by JSPS KAKENHI Grant Numbers 24K06717 and 24H00183.

\section{Ricci-flat K\"ahler metrics on $K3$ surfaces}\label{sec Kummer}
\subsection{Calabi-Yau manifolds}
\begin{definition}
\normalfont
Let $(X,\omega)$ be a K\"ahler manifold of dimension $m$ equipped with a closed $(m,0)$-form $\Omega$. 
$(X,\omega,\Omega)$ is a {\it Calabi-Yau manifold} if 
\begin{align*}
\omega^m=m!\left(\frac{\sqrt{-1}}{2}\right)^m(-1)^{m(m-1)/2}\Omega\wedge\bar{\Omega}.
\end{align*}
\end{definition}

\begin{rem}
\normalfont
By the definition, if $(X,\omega,\Omega)$ is a Calabi-Yau manifold, we can see that $\Omega$ is a holomorphic, nowhere vanishing $m$ form on $X$. It implies that the canonical bundle $K_X$ is trivial. Moreover, the K\"ahler metric induced by $\omega$ is Ricci-flat. 
\end{rem}
\begin{rem}
\normalfont
One of the typical examples of Calabi-Yau manifold is $\C^m$ with the standard K\"ahler form 
$\omega_0=\frac{\sqrt{-1}}{2}\sum_{k=1}^mdz_k\wedge d\bar{z}_k$ and the standard holomorphic volume form $\Omega_0=dz_1\wedge \cdots\wedge dz_m$. 
\end{rem}

In this paper, we only consider the case of dimension $2$. 
Let $(X,\omega)$ be a compact K\"ahler surface whose canonical bundle $K_X$ is holomorphically trivial. Let $\Omega$ be a nowhere vanishing $2$-form on $X$. Then there is a smooth function $f\in C^\infty(X)$ such that 
\begin{align*}
\omega^2=\frac{e^f}{2}\Omega\wedge \bar{\Omega}. 
\end{align*}
For a function $\varphi\in C^\infty(X)$ with 
\begin{align*}
\omega_\varphi:=\omega+\sqrt{-1}\ddb\varphi>0,
\end{align*}
consider the Monge-Amp\`ere equation 
\begin{align*}
\omega_\varphi^2=\frac{1}{2}\Omega\wedge\bar{\Omega}.
\end{align*}

In our definition, Calabi-Yau manifolds of complex  dimension $2$ are precisely hyper-K\"ahler manifolds of real dimension $4$. In fact, if $(X,\omega,\Omega)$ is a Calabi-Yau manifold of dimension $2$, then $\omega,{\rm Re}(\Omega),{\rm Im}(\Omega)$ give a hyper-complex structure $(I,J,K)$ and a Riemannian metric $g$ such that $\omega=g(I\cdot,\cdot)$, ${\rm Re}(\Omega)=g(J\cdot,\cdot)$, ${\rm Im}(\Omega)=g(K\cdot,\cdot)$. In this article, we use the term ``Calabi-Yau" rather than ``hyper-K\"ahler" since we always view $X$ as a complex manifold with respect to $I$ for the sake of convenience of the argument. 

\subsection{Kummer surfaces}\label{subsec kummer}
Here, we give examples of Calabi-Yau surfaces. 
Let $\Lambda\subset \C^2$ be a lattice of rank $4$ and $\Gamma\subset SU(2)$ a finite subgroup. Then $\Lambda$ acts on $\C^2$ by the translation and $\Gamma$ acts on $\C^2$ linearly. 
We assume 
\begin{align*}
\Gamma\cdot\Lambda\subset \Lambda,
\end{align*}
then $\Gamma$ acts on the $4$-torus $\T^4_\Lambda=\C^2/\Lambda$. Since the action is not necessarily free, the quotient space 
\begin{align*}
\T:=\T(\Lambda,\Gamma):=\T^4_\Lambda/\Gamma
\end{align*}
may have finitely many orbifold singularities. 

Let $(z,w)$ be the standard holomorphic coordinate of $\C^2$, and define the standard Calabi-Yau structure $(\omega_0,\Omega_0)$ on $\C^2$ by
\begin{align*}
\omega_0&:=\frac{\sqrt{-1}}{2}(dz\wedge d\bar{z}+dw\wedge d\bar{w}),\\
\Omega_0&:=dz\wedge dw.
\end{align*}
Since they are preserved by the action of $\Lambda$ and $\Gamma$, it descends to the Calabi-Yau structures on the smooth loci of $\C^2/\Gamma$, $\T^4_\Lambda$ and $\T$, respectively, we also denote them by $(\omega_0,\Omega_0)$ for the simplicity of the notations. 

Let $d_\T$ be the distance function on $\T$ defined by 
\begin{align*}
d_\T(x,y):=\inf\{ \| u-v\|\, |\, u\in x,\, v\in y\}
\end{align*}
for $x,y\in \T$, where we regard $x,y$ as $\Lambda\rtimes \Gamma$-orbits in $\C^2$. 
Here, $\|\cdot\|$ is the standard norm of $\C^2$. Since $\Gamma$ acts on $\C^2$ isometrically with respect to the norm, we can also define $\| x\|$ for $x\in\C^2/\Gamma$. Now, we put 
\begin{align*}
B_\Gamma(t)&:=\{ x\in\C^2/\Gamma|\, \| x\|<t\}/\Gamma\subset \C^2/\Gamma,\\
B_\T(x,t)&:=\{ y\in \T|\, d_\T(x,y)<t\}
\subset \T.
\end{align*}

Denote by $S$ the set of oribifold singularities of $\T$. 
For every $p\in S$ and sufficiently small $r>0$, there is a subgroup of $\Gamma$ denoted by 
$\Gamma_p$ such that there is a natural identification 
\begin{align*}
B_\T(p,r)\cong B_{\Gamma_p}(r). 
\end{align*}
If we take a lift $\tilde{p}\in\T^4_\Lambda$ of $p$, then 
\begin{align}
\Gamma_p=\{ a\in \Gamma|\, a\tilde{p}=\tilde{p}\}.\label{eq stab p}
\end{align}
Here, $\Gamma_p$ may depend on the choice of the lift of $p$; however, it is uniquely determined up to inner automorphisms.

There are a smooth complex surface $X=X(\Lambda,\Gamma)$ and a minimal resolution 
\begin{align*}
\pi=\pi_{\Lambda,\Gamma}\colon X\to \T,
\end{align*}
that is, a proper surjective map such that 
\begin{align*}
\pi|_{X\setminus \pi^{-1}(S)}\colon X\setminus \pi^{-1}(S)\to \T\setminus S
\end{align*}
is a biholomorphic map. Moreover, there is a unique holomorphic volume form $\Omega=\Omega_{\Lambda,\Gamma}\in\Omega^{2,0}(X)$ such that 
$\pi^*\Omega_0=\Omega$ on $X\setminus \pi^{-1}(S)$. 

\subsection{ALE spaces}\label{subsec ALE}
\begin{definition}
\normalfont
Let $(M,g)$ be a Riemannian manifold of dimension $4$ and $\Gamma$ is a finite subgroup of $O(4)$ whose linear action on $\R^4\setminus \{ 0\}$ is free. Then we say that $(M,g)$ is an {\it ALE space asymptotic to $\R^4/\Gamma$} if there are $R>0$, a compact set $K\subset M$ and a diffeomorphism $\Phi\colon (\R^4/\Gamma)\setminus \overline{B_\Gamma(R)}\to M\setminus K$ such that 
\begin{align*}
|\Phi^*g-g_0|_{g_0}&=O(r^{-4}),\\
|\nabla_0^k(\Phi^*g)|_{g_0}&=O(r^{-4-k})
\end{align*}
for all $k=1,2,\ldots$, where $g_0$ is the Euclidean metric, $\nabla_0$ is the Levi-Civita connection of $g_0$ and $r$ is the distance from the origin of $\R^4/\Gamma$. 
\label{def ALE}
\end{definition}
In \cite{kronheimer1989construction}, Kronheimer constructed a large family of hyper-K\"ahler manifold of dimension $4$ whose hyper-K\"ahler metric is ALE. Moreover, he showed in \cite{Kronheimer1989Torelli} that all of hyper-K\"ahler ALE spaces $(X,g)$ asymptotic to $\C^2/\Gamma$, where $\Gamma\subset SU(2)$,  can be constructed by \cite{kronheimer1989construction}.

In this paper, we only consider the hyper-K\"ahler ALE space given by the minimal resolution of $\C^2/\Gamma$. 
\begin{definition}
\normalfont
Let $(M,\eta,\Theta)$ be a Calabi-Yau manifold of complex dimension $2$, $\Gamma$ be a finite subgroup of $SU(2)$ and 
$\varpi\colon M\to \C^2/\Gamma$ is a surjective map. 
Assume that $\varpi$ is a minimal resolution of such that $\varpi^*\Omega_0=\Theta$ on $M\setminus\varpi^{-1}(\{ 0\})$. 
Denote by $g$ the K\"ahler metric of $\eta$. 
We call $(M,\eta,\Theta,\varpi)$ an {\it ALE Calabi-Yau manifold asymptotic to $\C^2/\Gamma$} if $(M,g)$ is an ALE space asymptotic to $\R^4/\Gamma$ with respect to $K=\varpi^{-1}(\overline{B_\Gamma(1)})$ and $\Phi=(\varpi|_{M\setminus K})^{-1}$ in Definition \ref{def ALE}. 
\end{definition}

\begin{rem}
\normalfont
In the ordinary situation, we do not assume $\Phi$ to be holomorphic. 
\end{rem}
\begin{rem}
\normalfont
For any minimal resolution $\varpi\colon M\to \C^2/\Gamma$, there exists a holomorphic $2$-form $\Theta$ such that $\Theta=\varpi^*\Omega_0$. Kronheimer constructed all of ALE hyper-K\"ahler metrics in \cite{kronheimer1989construction} as hyper-K\"ahler quotients. His family contains many ALE Calabi-Yau manifold asymptotic to $\C^2/\Gamma$ . Moreover, Joyce gave another construction of ALE Calabi-Yau manifolds in \cite{JoyceBook} and obtained the decay of the K\"ahler potentials. 
\end{rem}

\begin{thm}[{\cite[Theorem 8.2.3]{JoyceBook}}]
\label{thm ALE potential}
Let $\varpi\colon M\to \C^2/\Gamma$ be a minimal resolution and $M$ has a holomorphic $2$-form $\Theta$ such that $\Theta=\varpi^*\Omega_0$. For every K\"ahler class, there are a Ricci-flat K\"ahler form $\eta$ such that $(M,\eta,\Theta,\varpi)$ is an ALE Calabi-Yau manifold asymptotic to $\C^2/\Gamma$, and a function $\varphi\in C^\infty(M\setminus K)$ such that 
\begin{align*}
\eta|_{M\setminus K}&=\sqrt{-1}\ddb \varphi,\\
(\varpi^{-1})^*\varphi - \frac{r^2}{2}&=O(r^{-2}),\\
\left|
\nabla_0^k\left\{ (\varpi^{-1})^*\varphi- \frac{r^2}{2}\right\}
\right|_{g_0}&=O(r^{-2-k})
\quad (k=1,2,\ldots).
\end{align*}
\end{thm}
\begin{rem}
\normalfont
In \cite{JoyceBook}, the error term was estimated more precisely. In fact, there is a constant $A$ such that 
\begin{align*}
\left|
\nabla_0^k\left\{ (\varpi^{-1})^*\varphi- \frac{r^2}{2}-Ar^{-2}\right\}
\right|_{g_0}&=O(r^{\gamma-k})
\quad (k=0,1,2,\ldots)
\end{align*}
for some $-2<\gamma<-3$. 
\end{rem}

\begin{prop}
Let $(M,\eta,\Theta,\varpi)$ be an ALE Calabi-Yau manifold asymptotic to $\C^2/\Gamma$. There are a family of K\"ahler forms $(\eta_s)_{s>0}$ and holomorphic $\C^\times$-action ${\bf H}\colon \C^\times \to {\rm Aut}(M)$ such that $(M,\eta_s,\Theta)$ is a ALE Calabi-Yau manifold asymptotic to $\C^2/\Gamma$ for each $s>0$, $\eta_1=\eta$ and 
\begin{align*}
\varpi\circ {\bf H}_{\lambda}&=\lambda\cdot \varpi,\quad
{\bf H}_\lambda^*\Theta
=\lambda^2\Theta,\quad
{\bf H}_\lambda^*(\eta_{|\lambda|^2s})=|\lambda|^2\eta_s,
\end{align*}
for all $\lambda\in\C^\times$
\label{prop hk quotient}
\end{prop}
\begin{proof}
To prove the assertion, we briefly review the construction of ALE spaces in \cite{kronheimer1989construction}. 

From a finite subgroup $\Gamma\subset SU(2)$, we give a quaternionic vector space $\hat{M}$ with an inner product $g$ and a hyper-complex structure $(I,J,K)$, Lie group $G$ and $G$-action on $\hat{M}$. Then the action is trihamiltonian with respect to the hyper-K\"ahler structure on $\hat{M}$, and we have the hyper-K\"ahler moment map $\mu\colon \hat{M}\to \R^3\otimes \mathfrak{g}^*$, where $\mathfrak{g}$ is the Lie algebra of $G$. Denote by $Z\subset \mathfrak{g}^*$ the center of $\mathfrak{g}^*$. By \cite[Corollary 2.10]{kronheimer1989construction}, there are codimension one vector subspaces $D_1,\ldots,D_N\subset Z$ such that all 
\begin{align*}
\zeta\in R:=\R^3\otimes Z\setminus \bigcup_{j=1}^N(\R^3\otimes D_j)
\end{align*}
are regular values of $\mu$. If we take $\zeta\in R$, then $\mu^{-1}(\zeta)$ is a smooth submanifold of $\hat{M}$ and $G$ acts freely on $\mu^{-1}(\zeta)$, hence the quotient space $M_\zeta:=\mu^{-1}(\zeta)/G$ is a smooth manifold. Moreover, $(g,I,J,K)$ induces the hyper-K\"ahler structure $(g_\zeta,I_\zeta,J_\zeta,K_\zeta)$ on $M_\zeta$ naturally. If we put $\eta_\zeta=g_\zeta(I_\zeta\cdot,\cdot)$ and $\Theta_\zeta=g_\zeta(J_\zeta\cdot,\cdot)+\sqrt{-1}g_\zeta(K_\zeta\cdot,\cdot)$, then $(M_\zeta,\eta_\zeta,\Theta_\zeta)$ is a Calabi-Yau manifold of complex dimension $2$, with respect to $I_\zeta$. Now, put $\zeta=(\zeta_1,\zeta_2,\zeta_3)$ under the decomposition $\R^3\otimes \mathfrak{g}^*=\mathfrak{g}^*\oplus\mathfrak{g}^*\oplus\mathfrak{g}^*$. Then, $M_\zeta$ is biholomorphic to a minimal resolution of $\C^2/\Gamma$ iff $\zeta_2=\zeta_3=0$. In this case, $(M_\zeta,\eta_\zeta,\Theta_\zeta,\varpi_\zeta)$ become an ALE Calabi-Yau manifold asymptotic to $\C^2/\Gamma$ for some $\varpi_\zeta$. Now, let $(M,\eta,\Theta,\varpi)$ be any ALE Calabi-Yau manifolds asymptotic to $\C^2/\Gamma$. By the results in  \cite{Kronheimer1989Torelli}, we have 
\begin{align*}
(M,\eta,\Theta,\varpi)=(M_\zeta,\eta_\zeta,\Theta_\zeta,\varpi_\zeta)
\end{align*}
for some $\zeta_1\in Z\setminus (\bigcup_j D_j)$ and $\zeta=(\zeta_1,0,0)$. Then we can also see that $s\zeta_1\in Z\setminus (\bigcup_j D_j)$ for all $s>0$, we obtain a $1$-parameter family of ALE Calabi-Yau manifolds asymptotic to $\C^2/\Gamma$. 

Next, we write $\mu=(\mu_1,\mu_2,\mu_3)$ and put $\mu_\C:=\mu_2+\sqrt{-1}\mu_3$, then it is holomorphic. Denote by $G_\C$ the complexification of $G$. $G_\C$ acts on $\mu_\C^{-1}(0)$ holomorphically. Put 
\begin{align*}
\mu_\C^{-1}(0)_{\zeta_1}:=\{ x\in \mu_\C^{-1}(0)|\, (G_\C\cdot x)\cap \mu^{-1}(\zeta_1)\neq\emptyset\}.
\end{align*}
Then, by Kempf-Ness Theorem \cite{KempfNess1979}, the naturally induced map 
\begin{align}
M_{s\zeta}\to \mu_\C^{-1}(0)_{s\zeta_1}/G_\C\label{id KN}
\end{align}
is biholomorphic. 
Since the set $\mu_\C^{-1}(0)_{s\zeta_1}$ is independent of $s>0$ 
by \cite[Section 3]{Hattori2017TaubNUT}, we obtain the natural identification 
\begin{align}
M=M_\zeta\cong \mu_\C^{-1}(0)_{\zeta_1}/G_\C=\mu_\C^{-1}(0)_{s\zeta_1}/G_\C\label{id KN s}
\end{align}
for all $s$, as complex manifolds. Moreover, the Calabi-Yau structures $(\eta_{s\zeta},\Theta_{s\zeta})$ can be induced on $M$ by \eqref{id KN}\eqref{id KN s} and we have $\Theta=\Theta_{\zeta}=\Theta_{s\zeta}$ for all $s$ by the construction.

In the case of $\zeta=0$, $M_0$ is isometric to $\C^2/\Gamma$ with the standard metric and we have an isomorphism $M_0\to\mu_\C^{-1}(0)/\!\!/G_\C$, where the right-hand side is the GIT quotient. By composing \eqref{id KN} and the natural map $\mu_\C^{-1}(0)_{s\zeta_1}/G_\C\to \mu_\C^{-1}(0)/\!\!/G_\C$, we obtain 
\begin{align*}
\varpi_{s\zeta}\colon M_{s\zeta}\to \mu_\C^{-1}(0)/\!\!/G_\C\cong\C^2/\Gamma.
\end{align*}
Here, under the identifications \eqref{id KN}\eqref{id KN s}, we can see $\varpi_{s\zeta}=\varpi_\zeta=\varpi\colon M\to \C^2/\Gamma$. 
Next, we consider the natural $\C^\times$-action on $\hat{M}$ given by the scalar multiplication, that is, $(\lambda,x)\mapsto \lambda x$ for $\lambda\in\C^\times$ and $x\in\hat{M}$. 
This $\C^\times$-action commutes with the $G$-action. Moreover, for any $\lambda\in\C^\times$, we have $\mu_1(\lambda x)=|\lambda|^2\mu_1(x)$ and $\mu_\C(\lambda x)=\lambda^2\mu_\C(x)$. 
Consequently, multiplication by $\lambda$ descends to ${\bf H}_\lambda\colon \mu_\C^{-1}(0)_{\zeta_1}/G_\C\to \mu_\C^{-1}(0)_{|\lambda|^2\zeta_1}/G_\C$. 
Since $\mu_\C^{-1}(0)_{s\zeta_1}$ is independent of the choice of $s>0$, we obtain the biholomorphic map ${\bf H}_\lambda\colon M\to M$. Moreover, by the construction of $\varpi_{s\zeta},\eta_{s\zeta},\Theta_{s\zeta}$, we have $\varpi_{|\lambda|^2\zeta}\circ {\bf H}_{\lambda}(x)=\lambda\cdot \varpi_{\zeta}(x)$, ${\bf H}_\lambda^*\Theta_{|\lambda|^2\zeta}=\lambda^2\Theta_{\zeta}$ and ${\bf H}_\lambda^*\eta_{|\lambda|^2\zeta}=|\lambda|^2\eta_{\zeta}$. Here, we have $\varpi_{|\lambda|^2\zeta}=\varpi_\zeta=\varpi$ and $\Theta_{|\lambda|^2\zeta}=\Theta_\zeta$ for all $\lambda\in\C^\times$. Now, we put $\eta_s:=\eta_{s\zeta}$, then we have the assertion. See \cite{kronheimer1989construction} for more details. 
\end{proof}

\subsection{Approximating K\"ahler metrics on $X(\Lambda,\Gamma)$}\label{subsec appro}

Let $\Lambda,\Gamma,X,\Omega,\pi,S$ be as in Subsection \ref{subsec kummer}. 
Here, we construct the approximating K\"ahler metric following \cite{jiang2025kummer}. 
For $p\in S$ and $t>0$, we put 
\begin{align*}
U(p,t)
&:=\pi^{-1}(B_\T(p,t)).
\end{align*}
\begin{definition}
\normalfont
We write $(M_p,\eta_p,\Theta_p,\varpi_p)_{p\in S}\approx (\T,S)$ if every $(M_p,\eta_p,\Theta_p,\varpi_p)$ is an ALE Calabi-Yau manifold asymptotic to $\C^2/\Gamma_p$, where $\Gamma_p\subset \Gamma$ is given by \eqref{eq stab p}. 
\end{definition}
In this subsection, we suppose $(M_p,\eta_p,\Theta_p,\varpi_p)_{p\in S}\approx (\T,S)$. 
Let $\varphi_p$ be a function on the complement of some compact subset in $M_p$ given by Theorem \ref{thm ALE potential}. 
\begin{rem}
\normalfont
Let $(\eta_{p,s})_{s>0}$ be a one parameter family of K\"ahler forms obtained by applying Proposition \ref{prop hk quotient} to $(M_p,\eta_p,\Theta_p,\varpi_p)$. Since $\eta_{p,s}=s{\bf H}_{s^{-1/2}}^*\eta_p$ and $\eta_p|_{M_p\setminus K_p}=\sqrt{-1}\ddb\varphi_p$, we can see that $s{\bf H}_{s^{-1/2}}^*\varphi_p$ is a K\"ahler potential of $\eta_{p,s}$. 
\end{rem}

If $t$ is sufficiently small, we have the natural identification $B_\T(p,t)\cong B_{\Gamma_p}(t)$ and $U(p,t)\cong\varpi_p^{-1}(B_{\Gamma_p}(t))$. Composing this identification and ${\bf H}_{\varepsilon^{-1}}\colon \varpi_p^{-1}(B_{\Gamma_p}(t))\to\varpi_p^{-1}(B_{\Gamma_p}(\varepsilon^{-1}t))$, we have the biholomorphic map 
\begin{align*}
\mathcal{I}_{p,\varepsilon}\colon U(p,t)
&\stackrel{\cong}{\rightarrow} \varpi_p^{-1}(B_{\Gamma_p}(\varepsilon^{-1}t)),
\end{align*}
such that 
\begin{align*}
\pi-p&=\varepsilon \varpi_p\circ \mathcal{I}_{p,\varepsilon},\\
\Omega|_{U(p,t)}&=\varepsilon^2\mathcal{I}_{p,\varepsilon}^*\Theta_p.
\end{align*}
Moreover, we define a K\"ahler form on $X$ as follows. 
Let $\rho\colon \R\to \R$ be a $C^\infty$ function such that 
\begin{align*}
\rho|_{\{ t|\, t\ge 2-(1/10)\}}&\equiv 0,\quad
\rho|_{\{ t|\, t\le 1+(1/10)\}}\equiv 1,\quad 
0\le \rho\le 1.
\end{align*}
For $x\in\T$, put $d_\T(x,S):=\inf\{ d_\T(x,p)|\, p\in S\}$. Moreover, we define a smooth function $\chi_t\in C^\infty(X)$ by 
\begin{align*}
\chi_t(q)
:=\rho\left(\frac{d_\T(\pi(q),S)}{t}\right).
\end{align*}
Then, if $t>0$ is sufficiently small, $\chi_t$ satisfies 
\begin{align*}
{\rm supp}(\chi_t)&\subset \bigcup_{p\in S}U(p,2t),\\
\chi_t&\equiv 1\quad\mbox{ on } 
\bigcup_{p\in S}\overline{U(p,t)},\\
0&\le \chi_t\le 1.
\end{align*}

Now, put $r_p=d_\T(p,\cdot)$. 
Then $\omega_0=\sqrt{-1}\ddb r_p^2/2$ around $p$ and we define 
\begin{align*}
\omega_\varepsilon|_{X\setminus \bigcup_{p\in S}\overline{U(p,\varepsilon^{1/2})}}
&=\pi^*\omega_0
+\sqrt{-1}\ddb\left( \chi_{\varepsilon^{1/2}}\cdot\left( \varepsilon^2\mathcal{I}_{p,\varepsilon}^*\varphi_p 
- \frac{\pi^*r_p^2}{2}\right)\right),\\
\omega_\varepsilon|_{U(p,2\varepsilon^{1/2})}
&=\varepsilon^2\mathcal{I}_{p,\varepsilon}^*\eta_p
+\sqrt{-1}\ddb\left( (1-\chi_{\varepsilon^{1/2}})\left( \frac{\pi^*r_p^2}{2}
-\varepsilon^2\mathcal{I}_{p,\varepsilon}^*\varphi_p\right)\right).
\end{align*}
If we put $A(p,\varepsilon^{1/2}):=U(p,2\varepsilon^{1/2})\setminus \overline{ U(p,\varepsilon^{1/2})}$, we have 
\begin{align*}
\omega_\varepsilon|_{A(p,\varepsilon^{1/2})}
&=\sqrt{-1}\ddb\left( \chi_{\varepsilon^{1/2}}\cdot\varepsilon^2\mathcal{I}_{p,\varepsilon}^*\varphi_p + (1-\chi_{\varepsilon^{1/2}})\frac{\pi^*r_p^2}{2}\right).
\end{align*}
By Theorem \ref{thm ALE potential}, we have $|\nabla_0^k ( (\varpi_p^{-1})^*\varphi_p-r^2/2)|_{g_0}
=O(r^{-2-k})$, where $r$ is the distance from the origin of $\C^2/\Gamma_p$. We can also see 
\begin{align*}
\pi^*r_p=\varepsilon(\varpi_p\circ\mathcal{I}_{p,\varepsilon})^*r,
\end{align*}
therefore, 
\begin{align*}
\varepsilon^2 
\mathcal{I}_{p,\varepsilon}^*\varphi_p-\frac{\pi^*r_p^2}{2}
&=
\varepsilon^2 
\mathcal{I}_{p,\varepsilon}^*\varpi_p^*\left( (\varpi_p^{-1})^*\varphi_p-\frac{r^2}{2}\right).
\end{align*}

In this article, we are often interested in the behavior of several quantities as $\varepsilon\to 0$. If $f_\varepsilon,f'_\varepsilon$ are  nonnegative valued functions defined on the same domain with $f'_\varepsilon>0$, we write 
\begin{align*}
f_\varepsilon \lesssim f'_\varepsilon 
\end{align*}
if $\limsup_{\varepsilon\to 0}\sup_x \{ f_\varepsilon(x)/f'_\varepsilon(x)\}<\infty$. We often take $f'_\varepsilon$ as a constant function. Similarly, for constants $C_\varepsilon\ge 0$ and $C'_\varepsilon>0$, we write $C_\varepsilon\lesssim C'_\varepsilon$ if $\limsup_{\varepsilon\to 0}( C_\varepsilon/C'_\varepsilon)<\infty$. 

From now on, we often consider the flat metric $\pi^*g_0$ on $X\setminus \pi^{-1}(S)$ or $\varpi_p^*g_0$ on $M_p\setminus \varpi_p^{-1}(\{ 0\})$. We denote them by $g_0$ for simplicity of the notation, if there is no risk of confusion. 
\begin{lem}
Let $f$ be a smooth function on $\C^2/\Gamma_p\setminus \overline{B_{\Gamma_p}(R)}$ for some $R_1>0$. 
Assume $|\nabla_0^k f|_{g_0}=O(r^{-2-k})$. Then we have 
\begin{align*}
\left|\nabla_0^k \left\{\ddb\left( \chi_{\varepsilon^{1/2}}\cdot\varepsilon^2\mathcal{I}_{p,\varepsilon}^*\varpi_p^*f\right)\right\}\right|_{g_0}\lesssim \varepsilon^{2-k/2}
\end{align*}
on $A(p,\varepsilon^{1/2})$. 
\label{prop gluing region estimate}
\end{lem}
\begin{proof}
Since $(\varpi_p\circ\mathcal{I}_{p,\varepsilon})^*r^d|_{A(p,\varepsilon^{1/2})}\lesssim\varepsilon^{-d/2}$ and $(\mathcal{I}_{p,\varepsilon}^{-1})^*g_0=\varepsilon^2g_0$, we have 
\begin{align*}
\left|\nabla_0^k\left(
\varepsilon^2 
\mathcal{I}_{p,\varepsilon}^*\varpi_p^*f
\right)\right|_{g_0}
&=\varepsilon^2\left| \mathcal{I}_{p,\varepsilon}^*\varpi_p^*\nabla_0^k f\right|_{\varepsilon^2\mathcal{I}_{p,\varepsilon}^*\varpi_p^*g_0}\\
&=\varepsilon^{2-k}(\varpi_p\circ\mathcal{I}_{p,\varepsilon})^*\left| \nabla_0^k f\right|_{g_0}\lesssim \varepsilon^{3-k/2}.
\end{align*}
We also have $|\nabla_0^jdr_p|_{g_0}\lesssim r_p^{-j}$ on 
$A(p,\varepsilon^{1/2})$, hence 
\begin{align*}
\left|\nabla_0^k\chi_{\varepsilon^{1/2}}\right|_{g_0}\lesssim \varepsilon^{-k/2}
\mbox{ on }A(p,\varepsilon^{1/2}).
\end{align*}
Therefore, we obtain 
\begin{align*}
\left|\nabla_0^{k+2}\left( \chi_{\varepsilon^{1/2}}\cdot\varepsilon^2\mathcal{I}_{p,\varepsilon}^*\varpi_p^*f\right)\right|_{g_0}\lesssim \varepsilon^{2-k/2}
\end{align*}
which implies the assertion.
\end{proof}
Applying Lemma \ref{prop gluing region estimate} to $f=(\varpi_p^{-1})^*\varphi_p-r^2/2$, we obtain 
\begin{align}
\left|\nabla_0^k(\omega_\varepsilon -\pi^*\omega_0)\right|_{g_0}
&\lesssim \varepsilon^{2-k/2}
\mbox{ on }A(p,\varepsilon^{1/2}).
\label{eq decay gluing1}
\end{align}
In particular, we have 
\begin{align}
\left| \omega_\varepsilon|_{A(p,\varepsilon^{1/2})}^2
-\frac{1}{2}\Omega\wedge\bar{\Omega}\right|_{g_0}\lesssim \varepsilon^2.\label{ineq omega e - vol}
\end{align}

\subsection{Topology on $X$}
\begin{thm}
If $\Gamma\neq \{ 1\}$, then $X=X(\Lambda,\Gamma)$ is a $K3$ surface. 
\end{thm}
\begin{proof}
We have already shown that $X$ is a compact complex surface with a holomorphic volume form. It suffices to show that $X$ is simply-connected. First of all, the K\"ahler orbifold $\T$ always has a singularity since $\Gamma$ is nontrivial. Therefore, by the result of Zhang \cite[Theorem 2]{Zhang2026}, $\T$ is simply-connected. Then, it is easy to see that $\T\setminus S$ is also simply-connected. Moreover, all of the hyper-K\"ahler ALE spaces of dimension $4$ are simply connected. Hence $X$ is simply-connected by the van Kampen Theorem. 
\end{proof}

\section{Monge Amp\`ere equation}\label{sec MA}
In this section, we construct Ricci-flat K\"ahler metrics on $X=X(\Lambda,\Gamma)$. 
Let $(M_p,\eta_p,\Theta_p,\varpi_p)_{p\in S}\approx (\T,S)$ and $\omega_\varepsilon$ be the K\"ahler forms on $X$ defined in Subsection \ref{subsec appro}. 
The rigorous arguments for the gluing of the Ricci-flat K\"ahler metrics have been discussed in \cite{ArezzoPacard}, \cite{Donaldson2010}, \cite{jiang2025kummer}.
Here, we construct $\psi\in C^\infty(X)$ solving 
\begin{align*}
(\omega_\varepsilon+\sqrt{-1}\ddb\psi)^2
=\frac{1}{2}\Omega\wedge\bar{\Omega},
\end{align*}
where $\omega_\varepsilon$ is a K\"ahler form constructed in Subsection \ref{subsec appro}.

\subsection{Weighted H\"older spaces}\label{subsec weighted holder}
Let $\sigma_\varepsilon\colon X\to [\varepsilon,1]$ be a function defined by 
\[ \sigma_\varepsilon=
\left \{
\begin{array}{cc}
1 & 1\le d_\T(\cdot,S) \\
d_\T(\cdot,S) & \varepsilon\le d_\T(\cdot,S)\le 1\\
\varepsilon & d_\T(\cdot,S)\le\varepsilon
\end{array}
\right.
\]
Here, we do not need the smoothness of $\sigma_\varepsilon$.
Denote by $g_\varepsilon$ the K\"ahler metric of $\omega_\varepsilon$ $\nabla_{g_\varepsilon}$ the Levi-Civita connection, and ${\rm inj}_{g_\varepsilon}(x)$  the supremum of $r>0$ such that the exponential map $\exp_x\colon \{ v\in T_xX|\, |v|_{g_\varepsilon}<r\}\to X$ is embedding. 
By the construction of $g_\varepsilon$ in the previous subsection, as $\varepsilon\to 0$, the sectional curvature of $g_\varepsilon$ is bounded above by $O(\sigma_\varepsilon^{-2})$ and the limit of $g_\varepsilon$ as $\varepsilon\to 0$ is volume noncollapsing. Therefore, by the lower estimate for the injectivity radius by the sectional curvature and the volume growth by \cite{CheegerGromovTaylor1982}, we may suppose that there is $R>0$ such that ${\rm inj}_{g_\varepsilon}(x)\ge R\sigma_\varepsilon(x)$ for all $x$. 
Fixing $\varepsilon$, for $x,y\in X$, we write $x\sim y$ if $x\neq y$ and 
\begin{align*}
d_{g_\varepsilon}(x,y)
\le \min\left\{
R\sigma_\varepsilon(x),R\sigma_\varepsilon(y)\right\}.
\end{align*}
If $x\sim y$, then there exists a unique geodesic joining $x,y$. For such $x,y$ and a tensor $T$ on $X$, we can consider the difference $|T(x)-T(y)|$ compared by the parallel transportation of $\nabla_{g_\varepsilon}$. 

For $0<\alpha<1$ and $u\in C^k(X)$, put 
\begin{align*}
\| u\|_{C^k_{\delta,g_\varepsilon}}
&:=\sum_{j=0}^k\sup_x| \sigma_\varepsilon(x)^{-\delta+j}\nabla_{g_\varepsilon}^ju(x)|,\\
\| u\|_{C^{k,\alpha}_{\delta,g_\varepsilon}}
&:=\| u\|_{C^k_{\delta,g_\varepsilon}}\\
&\quad\quad
+\sup_{x\sim y} 
\min\left\{ \sigma_\varepsilon(x),\sigma_\varepsilon(y)\right\}^{-\delta+k+\alpha}
\frac{\left| \nabla_{g_\varepsilon}^ku(x)-\nabla_{g_\varepsilon}^ku(y)\right|}{d_{g_\varepsilon}(x,y)^\alpha},\\
\end{align*}
and denote by $C^k_{\delta,g_\varepsilon}(X)=C^k(X)$ and $C^{k,\alpha}_{\delta,g_\varepsilon}(X)=C^{k,\alpha}(X)$ the Banach spaces equipped with the above norms, respectively. 

\begin{lem}
Let $R$ be a constant appearing in the definition of $\| \cdot\|_{C^{k,\alpha}_{\delta,g_\varepsilon}}$. If we take $R>0$ sufficiently small, then there is a constant $C_{k,\delta}>0$ independent of $\varepsilon$  such that $\| u\|_{C^{k,\alpha}_{\delta,g_\varepsilon}}\le C_{k,\delta}\| u\|_{C^{k+1}_{\delta,g_\varepsilon}}$ for all $u\in C^{k+1}(X)$. 
\label{lem k,a < k+1}
\end{lem}
\begin{proof}
First of all, we show that there are constants $N,R,\varepsilon_0>0$ independent of $\varepsilon$ such that $\sigma_\varepsilon|_{B_\varepsilon(x,R\sigma_\varepsilon(x))}\ge \sigma_\varepsilon(x)/N$ for all $x$ and $\varepsilon\in(0,\varepsilon_0]$, where $B_\varepsilon(x,t)$ is the open geodesic ball with respect to $g_\varepsilon$. Fix a sufficiently small $t>0$ such that $U(p,t)\cap U(q,t)=\emptyset$ for $p\neq q$. By the definition of ALE spaces, there is $N'\ge 1$ such that $\frac{1}{2}\varpi_p^*\omega_0\le \eta_p\le 2\varpi_p^*\omega_0$ on $\varpi_p^{-1}(\C^2\setminus B_{\Gamma_p}(N'))$, for all $p\in S$. Combining with \eqref{eq decay gluing1}, there is $\varepsilon_0>0$ such that 
$\frac{1}{2}\pi^*\omega_0\le \omega_\varepsilon\le 2\pi^*\omega_0$ on $X_1:=X\setminus \bigcup_{p\in S}\overline{U(p,N'\varepsilon)}$ for all $\varepsilon\in(0,\varepsilon_0]$. Now, we identify $X\setminus \pi^{-1}(S)$ with $\T\setminus S$ by $\pi$. Take $x\in X_2:=X\setminus \bigcup_{p\in S}\overline{U(p,2N'\varepsilon)}$ and $y\in B_\varepsilon(x,\sigma_\varepsilon(x)/3)$. Suppose $c\colon [0,1]\to X$ is the geodesic joining $x$ and $y$. Here, we show ${\rm Im}(c)\subset X_1$. Suppose not, then there are $s_0$ and $p\in S$ such that $d_\T(c(s_0),p)=N'\varepsilon$ and 
\begin{align*}
\frac{\sigma_\varepsilon(x)}{3}>d_{g_\varepsilon}(x,y)\ge d_{g_\varepsilon}(x,c(s_0))
&\ge \frac{1}{\sqrt{2}}d_\T(x,c(s_0))\\
&\ge \frac{1}{\sqrt{2}}(d_\T(p,x)-N'\varepsilon).
\end{align*}
Since $\sigma_\varepsilon(x)\le d_\T(x,p)$, we have $d_\T(p,x)<3N'\varepsilon/(3-\sqrt{2})$, which contradicts $d_\T(p,x)\ge 2N'\varepsilon$. Therefore, we obtain ${\rm Im}(c)\subset X_1$. Then we can show 
\begin{align*}
\frac{\sigma_\varepsilon(x)}{3}>d_{g_\varepsilon}(x,y)\ge \frac{1}{\sqrt{2}}d_\T(x,y)
&\ge \frac{1}{\sqrt{2}}|\sigma_\varepsilon(x)-\sigma_\varepsilon(y)|,
\end{align*}
which implies $\sigma_\varepsilon(y)>(1-\sqrt{2}/3)\sigma_\varepsilon(x)$. Thus we obtain $\sigma_\varepsilon|_{B_\varepsilon(x,\sigma_\varepsilon(x)/3)}\ge (1-\sqrt{2}/3)\sigma_\varepsilon(x)$, if $\varepsilon\le \varepsilon_0$ and $x\in X_2$. If $x\in \overline{U(p,2N'\varepsilon)}$, then $\sigma_\varepsilon|_{B_\varepsilon(x,\sigma_\varepsilon(x)/3)}\ge \varepsilon\ge \sigma_\varepsilon(x)/{2N'}$. Therefore, if we take $N=\max\{ 2N',(1-\sqrt{2}/3)^{-1}\}$ and $R\le 1/3$, we have $\sigma_\varepsilon|_{B_\varepsilon(x,R\sigma_\varepsilon(x))}\ge \sigma_\varepsilon(x)/N$ for all $x$ and $\varepsilon\in(0,\varepsilon_0]$.

Let $c\colon [0,d_{g_\varepsilon}(x,y)]\to X$ be the geodesic with unit speed and $c(0)=x$, $c(d_{g_\varepsilon}(x,y))=y$. Then we have  
\begin{align*}
\left| \nabla_{g_\varepsilon}^ku(x)-\nabla_{g_\varepsilon}^ku(y)\right|
&\le \int_0^{d_{g_\varepsilon}(x,y)}|\nabla_{g_\varepsilon}^{k+1}u(c(t))|dt,
\end{align*}
and 
\begin{align*}
&\quad\,\,
\frac{\min\left\{ \sigma_\varepsilon(x),\sigma_\varepsilon(y)\right\}^{-\delta+k+1}}{N^{-\delta+k+1}}\left| \nabla_{g_\varepsilon}^ku(x)-\nabla_{g_\varepsilon}^ku(y)\right|\\
&\le \int_0^{d_{g_\varepsilon}(x,y)}\sigma_\varepsilon(c(t))^{-\delta+k+1}|\nabla_{g_\varepsilon}^{k+1}u(c(t))|dt\\
&\le d_{g_\varepsilon}(x,y)\sup\left|\sigma_\varepsilon^{-\delta+k+1}\nabla_{g_\varepsilon}^{k+1}u\right|.
\end{align*}
Therefore, if $x\sim y$, we obtain 
\begin{align*}
&\quad\,\,\min\left\{ \sigma_\varepsilon(x),\sigma_\varepsilon(y)\right\}^{-\delta+k+\alpha}
\frac{\left| \nabla_{g_\varepsilon}^ku(x)-\nabla_{g_\varepsilon}^ku(y)\right|}{d_{g_\varepsilon}(x,y)^\alpha}\\
&\le N^{-\delta+k+1}\min\left\{ \sigma_\varepsilon(x),\sigma_\varepsilon(y)\right\}^{\alpha-1}d_{g_\varepsilon}(x,y)^{1-\alpha}\sup\left|\sigma_\varepsilon^{-\delta+k+1}\nabla_{g_\varepsilon}^{k+1}u\right|\\
&\le N^{-\delta+k+1}R^{1-\alpha}\sup\left|\sigma_\varepsilon^{-\delta+k+1}\nabla_{g_\varepsilon}^{k+1}u\right|.
\end{align*}
It gives $\| u\|_{C^{k,\alpha}_{\delta,g_\varepsilon}}\le N^{-\delta+k+1}R^{1-\alpha}\| u\|_{C^{k+1}_{\delta,g_\varepsilon}}$.
\end{proof}

Define an operator $L_{\omega_\varepsilon}\colon C^{k+2,\alpha}_{\delta,g_\varepsilon}(X)\to C^{k,\alpha}_{\delta-2,g_\varepsilon}(X)$ by 
\begin{align*}
L_{\omega_\varepsilon}(\psi)
:=\frac{2\sqrt{-1}\ddb\psi\wedge \omega_\varepsilon}{\omega_\varepsilon^2}.
\end{align*}
Put 
\begin{align*}
E^{k,\alpha}_{\delta,g_\varepsilon}
:=\left\{ \psi\in C^{k,\alpha}_{\delta,g_\varepsilon}(X)\left|\, 
\int_X\psi \omega_\varepsilon^2=0
\right.\right\},
\end{align*}
then it is a closed subspace of $C^{k,\alpha}_{\delta,g_\varepsilon}(X)$. Moreover, we obtain a bounded operator 
$L_{\omega_\varepsilon}\colon E^{k+2,\alpha}_{\delta,g_\varepsilon}\to E^{k,\alpha}_{\delta-2,g_\varepsilon}$. Then we have the next theorem. 
\begin{thm}[{\cite[Theorem 3.6.8]{jiang2025kummer}}]
\label{thm schauder}
Let $\delta\in (-2,0)$. 
Then $L_{\omega_\varepsilon}\colon E^{k+2,\alpha}_{\delta,g_\varepsilon}\to E^{k,\alpha}_{\delta-2,g_\varepsilon}$ is an invertible bounded operator and 
there is a constant $K_{k,\delta}>0$ independent of $\varepsilon$ such that for all $u\in E^{k+2,\alpha}_{\delta,g_\varepsilon}$ and sufficiently small $\varepsilon>0$, 
\begin{align*}
\| u\|_{C^{k+2,\alpha}_{\delta,g_\varepsilon}}\le K_{k,\delta}\| L_{\omega_\varepsilon}(u)\|_{C^{k,\alpha}_{\delta-2,g_\varepsilon}}
\end{align*}
\end{thm}

\subsection{Construction of solutions}
Let $\omega_\varepsilon$, $\Omega$ be as in Subsection \ref{subsec appro}. 
Define $f_\varepsilon\in C^\infty(X)$ by 
\begin{align*}
\omega_\varepsilon^2=\frac{e^{f_\varepsilon}}{2}\Omega\wedge\bar{\Omega}
\end{align*}
and put 
\begin{align*}
W_\varepsilon:=\frac{\int_Xe^{f_\varepsilon}\Omega\wedge\bar{\Omega}}{\int_X\Omega\wedge\bar{\Omega}}.
\end{align*}
By the construction of $\omega_\varepsilon$, the support of $f_\varepsilon$ is contained in $\bigcup_{p\in S}A(p,\varepsilon^{1/2})$, and by the estimate \eqref{ineq omega e - vol}, we have $|W_\varepsilon-1|\lesssim \varepsilon^4$. 

Now, we have 
\begin{align*}
(\omega_\varepsilon+\sqrt{-1}\ddb\psi)^2
=(1+L_{\omega_\varepsilon}\psi)\omega_\varepsilon^2
-(\ddb\psi)^2.
\end{align*}
In this subsection, we construct a function $\varphi_\varepsilon\in C^\infty(X)$ satisfying the equation  
\begin{align}
(\omega_\varepsilon+\sqrt{-1}\ddb\varphi_\varepsilon)^2
=W_\varepsilon e^{-f_\varepsilon}\omega_\varepsilon^2.\label{eq MA}
\end{align}
In the following argument, we apply Theorem \ref{thm schauder}. 
\begin{prop}\label{prop alm ricci flat}
We have 
\begin{align*}
\|W_\varepsilon e^{-f_\varepsilon}-1\|_{C^{k,\alpha}_{\delta-2,g_\varepsilon}}
&\lesssim \varepsilon^{3-\delta/2}.
\end{align*}
\end{prop}
\begin{proof}
We may write 
\begin{align*}
w:=W_\varepsilon e^{-f_\varepsilon}-1
=e^{-f_\varepsilon}-1+(W_\varepsilon-1)e^{-f_\varepsilon}.
\end{align*}
By the estimate \eqref{eq decay gluing1}, for $k\ge 0$, 
we have 
\begin{align*}
|\nabla_0^k(e^{-f_\varepsilon}-1)|_{g_0}
&\lesssim\varepsilon^{2-k/2}\mbox{ on }A(p,\varepsilon^{1/2}),
\end{align*}
which gives 
\begin{align*}
|\nabla_{g_\varepsilon}^k(e^{-f_\varepsilon}-1)|_{g_\varepsilon}
&\lesssim\varepsilon^{2-k/2}\mbox{ on }A(p,\varepsilon^{1/2}).
\end{align*}
Since $|W_\varepsilon-1|\lesssim\varepsilon^4$, we obtain 
\begin{align*}
|\nabla_{g_\varepsilon}^kw|_{g_\varepsilon}
&\lesssim\varepsilon^{2-k/2}\mbox{ on }A(p,\varepsilon^{1/2}),\\
|\nabla_{g_\varepsilon}^kw|_{g_\varepsilon}
&\equiv 0\mbox{ on }X\setminus\bigcup_{p\in S}A(p,\varepsilon^{1/2})\quad (k\ge 1).
\end{align*}
Since $\sigma_\varepsilon< 2\varepsilon^{1/2}$ on $A(p,\varepsilon^{1/2})$, we obtain 
\begin{align*}
\sum_{j=0}^k\sup\sigma_\varepsilon^{-(\delta-2)+j}|\nabla_{g_\varepsilon}^jw|_{g_\varepsilon}
&\lesssim\sum_{j=0}^k\varepsilon^{2-j/2+(-(\delta-2)+j)/2}\lesssim\varepsilon^{3-\delta/2}
\end{align*}
for all $k$. By Lemma \ref{lem k,a < k+1}, we have $\| w\|_{C^{k,\alpha}_{\delta-2,g_\varepsilon}}\lesssim\varepsilon^{3-\delta/2}$. 
\end{proof}

\begin{thm}
There is $\varphi_\varepsilon\in C^\infty(X)$ satisfying \eqref{eq MA} and $\|\varphi_\varepsilon\|_{C^{k+2,\alpha}_{\delta,g_\varepsilon}}\lesssim\varepsilon^{3-\delta/2}$ for every $k\ge 0$ and $-2<\delta<0$. 
\label{thm fix pt}
\end{thm}
\begin{proof}
The standard argument of the fixed-point theorem gives the proof. We construct a sequence $(\psi_n)_{n=0}^\infty$ inductively by 
\begin{align*}
L_{\omega_\varepsilon}\psi_{n+1}
&=W_\varepsilon e^{-f_\varepsilon}-1
+\frac{(\ddb\psi_n)^2}{\omega_\varepsilon^2},
\end{align*}
since $L_{\omega_\varepsilon}$ is invertible. We will show that $(\psi_n)_n$ is a Cauchy sequence and the limit is the solution of \eqref{eq MA}. 

Let $-2<\delta<0$ and $S,T$ be a tensor on $X$. 
Then by \cite[Proposition 3.4.5]{jiang2025kummer}, we have 
\begin{align*}
\| S\otimes T\|_{C^{0,\alpha}_{\delta-2,g_\varepsilon}}
&\le \varepsilon^{\delta-2}
\| S\|_{C^{0,\alpha}_{\delta-2,g_\varepsilon}}\| T\|_{C^{0,\alpha}_{\delta-2,g_\varepsilon}}.
\end{align*}

By a similar argument, we can see 
\begin{align*}
\| S\otimes T\|_{C^{k,\alpha}_{\delta-2,g_\varepsilon}}
&\le\varepsilon^{\delta-2}
\sum_{l=0}^k\frac{k!}{l!(k-l)!}\| S\|_{C^{l,\alpha}_{\delta-2,g_\varepsilon}}\| T\|_{C^{k-l,\alpha}_{\delta-2,g_\varepsilon}}.
\end{align*}
Therefore, there is a constant $C_k>0$ such that 
\begin{align*}
&\quad\quad
\left\|\frac{(\ddb\psi_n)^2-(\ddb\psi_{n-1})^2}{\omega_\varepsilon^2}\right\|_{C^{k,\alpha}_{\delta-2,g_\varepsilon}}\\
&\le C_k\varepsilon^{\delta-2}\sum_{l=0}^k
\left( \|\ddb\psi_n\|_{C^{l,\alpha}_{\delta-2,g_\varepsilon}}
+ \| \ddb\psi_{n-1}\|_{C^{l,\alpha}_{\delta-2,g_\varepsilon}}
\right)
\|\ddb\psi_n- \ddb\psi_{n-1}\|_{C^{k-l,\alpha}_{\delta-2,g_\varepsilon}}.
\end{align*}
Then, by Theorem \ref{thm schauder}, we have 
\begin{align*}
\| \psi_{n+1}-\psi_n\|_{C_{\delta,g_\varepsilon}^{k+2,\alpha}}
&\le K_k
\left\|\frac{(\ddb\psi_n)^2-(\ddb\psi_{n-1})^2}{\omega_\varepsilon^2}\right\|_{C^{k,\alpha}_{\delta-2,g_\varepsilon}}\\
&\le K_kC_k\varepsilon^{\delta-2}
\sum_{l=0}^{k}\left(\| \psi_n\|_{C_{\delta,g_\varepsilon}^{l+2,\alpha}}
+\| \psi_{n-1}\|_{C_{\delta,g_\varepsilon}^{l+2,\alpha}}
\right)\| \psi_n-\psi_{n-1}\|_{C_{\delta,g_\varepsilon}^{k-l+2,\alpha}}.
\end{align*}
We also have 
\begin{align*}
\| \psi_{n+1}\|_{C_{\delta,g_\varepsilon}^{k+2,\alpha}}
&\le K_k\| W_\varepsilon e^{-f_\varepsilon}-1\|_{C_{\delta-2,g_\varepsilon}^{k,\alpha}}\\
&\quad\quad
+K_kC_k\varepsilon^{\delta-2}
\sum_{l=0}^k\| \psi_n\|_{C_{\delta,g_\varepsilon}^{l+2,\alpha}}\| \psi_n\|_{C_{\delta,g_\varepsilon}^{k-l+2,\alpha}}.
\end{align*}

From now on, we put $\psi_0=0$ and 
\begin{align*}
s_k&:= 2\max\left\{ \left.
K_l\|We^{-f}-1\|_{C^{l,\alpha}_{\delta-2,g_\varepsilon}}
\,\right|\, l=0,\ldots,k\right\}.
\end{align*}
By Proposition \ref{prop alm ricci flat}, we have $s_k\le C_k'\varepsilon^{3-\delta/2}$ for some constant $C_k'>0$. Put 
\begin{align*}
\varepsilon_{k,\delta}
:=\left(\frac{1}{4(k+1)C_kK_kC_k'}\right)^{2/(2+\delta)},
\end{align*}
then we have $K_kC_k\varepsilon^{\delta-2}(k+1)s_k\le \frac{1}{4}$ for all $\varepsilon\le \varepsilon_{k,\delta}$. 
If $\| \psi_n\|_{C_{\delta,g_\varepsilon}^{l+2,\alpha}}\le s_l$ for all $l\le k$ and $\varepsilon\le \min\{ \varepsilon_{0,\delta},\ldots,\varepsilon_{k,\delta}\}$, 
we have  
\begin{align*}
\| \psi_{n+1}\|_{C_{\delta,g_\varepsilon}^{k+2,\alpha}}
&\le K_k\| W_\varepsilon e^{-f_\varepsilon}-1\|_{C_{\delta-2,g_\varepsilon}^{k,\alpha}}
+K_kC_k\varepsilon^{\delta-2}
\sum_{l=0}^ks_ls_{k-l}\\
&\le \frac{s_k}{2}
+K_kC_k\varepsilon^{\delta-2}
(k+1)s_k^2
\le s_k.
\end{align*}
Therefore, by the induction for $k,n$, we have 
$\| \psi_n\|_{C_{\delta,g_\varepsilon}^{k+2,\alpha}}\le s_k$ if $\varepsilon\le \min\{ \varepsilon_{0,\delta},\ldots,\varepsilon_{k,\delta}\}$, for all $n$ and $k$. 
Therefore, we have 
\begin{align*}
\| \psi_{n+1}-\psi_n\|_{C_{\delta,g_\varepsilon}^{k+2,\alpha}}
&\le K_kC_k\varepsilon^{\delta-2}
\sum_{l=0}^{k}\left(\| \psi_n\|_{C_{\delta,g_\varepsilon}^{l+2,\alpha}}
+\| \psi_{n-1}\|_{C_{\delta,g_\varepsilon}^{l+2,\alpha}}
\right)\| \psi_n-\psi_{n-1}\|_{C_{\delta,g_\varepsilon}^{k-l+2,\alpha}}\\
&\le K_kC_k\varepsilon^{\delta-2}
\sum_{l=0}^{k}2s_l\| \psi_n-\psi_{n-1}\|_{C_{\delta,g_\varepsilon}^{k-l+2,\alpha}}.
\end{align*}
If we put $\Psi_{n,k}:=\max\{ \| \psi_n-\psi_{n-1}\|_{C_{\delta,g_\varepsilon}^{l+2,\alpha}}|\, l=0,\ldots,k\}$, we have 
\begin{align*}
\Psi_{n+1,k}
&\le 2(k+1)s_kK_kC_k\varepsilon^{\delta-2}
\Psi_{n,k}.
\end{align*}
Again by $K_kC_k\varepsilon^{\delta-2}(k+1)s_k\le \frac{1}{4}$, we obtain $\Psi_{n+1,k}
\le \frac{1}{2}\Psi_{n,k}$, which implies that $(\psi_n)_n$ is a Cauchy sequence in $C^{k+2,\alpha}_{\delta,g_\varepsilon}(X)$. 
Then $(\psi_n)_n$ converges to $\varphi_\varepsilon$ in $C^{k+2,\alpha}_{\delta,g_\varepsilon}$, which solves \eqref{eq MA}. Since $(\psi_n)_n$ is independent of $k,\delta$, its limit $\varphi_\varepsilon$ is also independent of $k,\delta$. Consequently, it is $C^\infty$-class. Moreover, the estimate $\| \psi_n\|_{C_{\delta,g_\varepsilon}^{k+2,\alpha}}\le s_k$ implies $\| \varphi_\varepsilon\|_{C_{\delta,g_\varepsilon}^{k+2,\alpha}}\le s_k\lesssim \varepsilon^{3-\delta/2}$. 
\end{proof}

\subsection{Proof of Theoreom \ref{thm schauder}}
The argument in the previous subsection is based on Theorem \ref{thm schauder}, which was proved in \cite{jiang2025kummer}. However, some of the assumptions and settings adopted here differ slightly from those in \cite{jiang2025kummer}. Although the proof is essentially the same, we provide an outline in this subsection for the reader\textquotesingle s convenience.

In this subsection, let $X=X(\Lambda,\Gamma)$ be as in Subsection \ref{subsec kummer}. Following the notations in Subsection \ref{subsec appro}, we put 
\begin{align*}
X_t&:=X\setminus \bigcup_{p\in S} \overline{U(p,t)},\\
Y_{p,s,t}&:=U(p,s)\setminus \overline{U(p,t)}\quad (s>t),\\
Z_t&:=\bigcup_{p\in S} U(p,t).
\end{align*}
Here, we take $s,t>0$ to be smaller than 
\begin{align*}
t_0:=\frac{1}{4}\inf\left\{ d_\T(p,q)|\, p,q\in S,\, p\neq q\right\}>0.
\end{align*}
The K\"ahler form $\omega_\varepsilon$ defined in Subsection \ref{subsec appro} satisfies the following estimates; 
\begin{align*}
\omega_\varepsilon&=\pi^*\omega_0\mbox{ on }X_{2\varepsilon^{1/2}},\\
\varepsilon^{-2}(\mathcal{I}_{p,\varepsilon}^{-1})^*\omega_\varepsilon
&=\eta_p\mbox{ on }\mathcal{I}_{p,\varepsilon}^{-1}(U(p,\varepsilon^{1/2})),\\
|\nabla_0^k(\varepsilon^{-2}(\mathcal{I}_{p,\varepsilon}^{-1})^*\omega_\varepsilon-\eta_p)|_{g_p}
&=O(r^{-4-k})\mbox{ on } \mathcal{I}_{p,\varepsilon}^{-1}(Y_{p,t_0,\varepsilon^{1/2}}),
\end{align*}
and we also have the same estimates for $g_\varepsilon-\pi^*g_0$ and $\varepsilon^{-2}(\mathcal{I}_{p,\varepsilon}^{-1})^*g_\varepsilon-g_p$.
By the above estimates, we can see that the K\"ahler metrics $g_\varepsilon$ of $\omega_\varepsilon$ satisfy 
\begin{align}
\lim_{\varepsilon\to 0}\sup_{X_t}|\nabla_0^k(g_\varepsilon-\pi^*g_0)|_{g_0}&=0,\label{eq asymp1}\\
\lim_{\varepsilon\to 0}\sup_{\mathcal{I}_{p,\varepsilon}^{-1}(U(p,t))}(\max\{ 1,\varpi_p^*r\})^k|\nabla^k_{g_p}(\varepsilon^{-2}(\mathcal{I}_{p,\varepsilon}^{-1})^*g_\varepsilon-g_p)|_{g_p}&=0\label{eq asymp2}
\end{align}
for every $0<t\le t_0$ and $k\ge 0$. In this subsection, we only assume that $(g_\varepsilon)_\varepsilon$ is a family of Riemannian metrics on $X$ satisfying \eqref{eq asymp1}\eqref{eq asymp2}, not necessarily to be the K\"ahler metric defined in Subsection \ref{subsec appro}. 


\begin{prop}[{Weighted Schauder estimate}]
Let $g_\varepsilon$ be Riemannian metrics on $X=X(\Lambda,\Gamma)$ satisfying \eqref{eq asymp1}\eqref{eq asymp2}.
There are positive constants $\varepsilon_0$ and $C_{k,\delta}$ independent of $\varepsilon$ such that 
\begin{align*}
\| u\|_{C^{k+2,\alpha}_{\delta,g_\varepsilon}}
\le C_{k,\delta}(\| \Delta_{g_\varepsilon}(u)\|_{C^{k,\alpha}_{\delta-2,g_\varepsilon}}
+\| u\|_{C^0_{\delta,g_\varepsilon}})
\end{align*}
for all $0<\varepsilon\le\varepsilon_0$ and $u\in C^{k+2,\alpha}_{\delta,g_\varepsilon}(X)$. 
\label{prop weight sch}
\end{prop}
The proof is based on the standard Schauder estimate.
\begin{thm}[{\cite{GilbargTrudinger}, Corollary 6.3 and Problem 6.1}]
Let $\Omega\subset \R^n$ be open and bounded, 
$\Omega'\subset \Omega$ be open such that 
$\overline{\Omega}'\subset \Omega$, 
$u\in C^{k+2,\alpha}(\Omega)$ be a bounded solution of 
\begin{align*}
Lu:=a^{ij}\del_i\del_ju+b^i\del_iu+cu=f,
\end{align*}
where $a^{ij}=a^{ji},b^i,c,f\in C^{k,\alpha}(\Omega)$. Suppose there are constants 
$\lambda,\Lambda,d,L>0$ such that 
$(a^{ij})_{i,j}\ge \lambda$ on $\Omega$ and 
\begin{align*}
&\max\left\{ \| a^{ij}\|_{C^{k,\alpha}(\Omega)},\,
\|b^i\|_{C^{k,\alpha}(\Omega)},\,
\|c\|_{C^{k,\alpha}(\Omega)}\right\} \le \Lambda,\\
&{\rm dist}(\Omega',\del\Omega)\ge d.
\end{align*}
Then there is a constant $C=C(n,k,\alpha,\lambda,\Lambda,d)$ such that 
\begin{align*}
\| u\|_{C^{k+2,\alpha}(\Omega')}
&\le C\left( \| f\|_{C^{k,\alpha}(\Omega)}+\| u\|_{C^0(\Omega)}
\right).
\end{align*}
\end{thm}

As a consequence of the standard Schauder estimate, we have the next estimate. 
\begin{cor}
Let $(M,g)$ be a Riemannian manifold and $\Omega'\subset \Omega\subset M$ be open sets such that $\overline{\Omega'}\subset \Omega$ and $\overline{\Omega'}$ is compact. 
Denote by $\Delta_g$ the Laplace-Beltrami operator of $g$. Let $A_0\ge 1,A_1,A_2\ldots$ be positive constants  and denote by $\mathfrak{M}(g,\Omega,(A_k)_{k=0}^\infty)$ the set consisting of Riemannian metrics $h$ on $M$ with 
\begin{align*}
A_0^{-1}g|_{\Omega}\le h|_{\Omega}\le A_0g|_{\Omega},\quad |\nabla_g^kh|_{\Omega}|\le A_k.
\end{align*}
Then there is a constant $C_k=C(g,\Omega',\Omega,A_0,\ldots,A_k)$ such that 
\begin{align*}
\| u\|_{C^{k+2,\alpha}_h(\Omega')}
&\le C_k\left( \| \Delta_h u\|_{C^{k,\alpha}_h(\Omega)}+\| u\|_{C^0(\Omega)}
\right)
\end{align*}
for all $h\in\mathfrak{M}(g,\Omega,(A_k)_{k=0}^\infty)$ and $u\in C^{k+2,\alpha}_h(\Omega')$. 
\label{cor schauder local}
\end{cor}

On the Riemannian manifold $(X,g)$ and a positive valued function $\sigma\in C^\infty(X)$, define the weighted H\"older norm $\|\cdot \|_{C^{k,\alpha}_{\delta,\sigma,g}}$ with respect to the weight function $\sigma$ similarly as in Subsection \ref{subsec weighted holder}. 
Note that the ordinary H\"older norm can be written as $\|\cdot\|_{C_g^{k,\alpha}}=\|\cdot\|_{C^{k,\alpha}_{\delta,1,g}}$ for every $\delta\in\R$. It is easy to check that $\|\cdot \|_{C^{k,\alpha}_{\delta,\sigma,\lambda^2g}}=\lambda^{-\delta}\|\cdot \|_{C^{k,\alpha}_{\delta,\sigma/\lambda,g}}$ for a positive constant $\lambda$. 

\begin{proof}[{Proof of Proposition \ref{prop weight sch}}]
By \eqref{eq asymp1}, the family $\{ g_\varepsilon\}_{0<\varepsilon\ll 1}$ is included in $\mathfrak{M}(\pi^*g_0,X_{t_0/4},(A_k)_k)$ for some constants $A_k$. 
Then Corollary \ref{cor schauder local} implies that there is $C_1>0$ independent of $\varepsilon$ such that 
\begin{align*}
\| u\|_{C^{k+2,\alpha}_{\delta,1,g_\varepsilon}(X_{t_0/2})}
&\le C_1\left( \| \Delta_{g_\varepsilon} u\|_{C^{k,\alpha}_{\delta-2,1,g_\varepsilon}(X_{t_0/4})}
+\| u\|_{C^0_{\delta,1,g_\varepsilon}(X_{t_0/4})}
\right)
\end{align*}
for all $u\in C^{k+2,\alpha}_{\delta,1,g_\varepsilon}(X_{t_0/4})$. 

By \eqref{eq asymp2}, $\{ h^p_\varepsilon:=\varepsilon^{-2}(\mathcal{I}_{p,\varepsilon}^{-1})^*g_\varepsilon\}_{0<\varepsilon\ll 1}$ is included in $\mathfrak{M}(g_p,\varpi_p^{-1}(B_{\Gamma_p}(4)),(A_k)_k)$ for some constants $A_k$. Accordingly, we have a constant $C_2>0$ such that 
\begin{align*}
\| u\|_{C^{k+2,\alpha}_{\delta,1,h^p_\varepsilon}(\varpi_p^{-1}(B_{\Gamma_p}(3)))}
&\le C_2\left( \| \Delta_{h^p_\varepsilon} u\|_{C^{k,\alpha}_{\delta-2,1,h^p_\varepsilon}(\varpi_p^{-1}(B_{\Gamma_p}(4)))}
+\| u\|_{C^0_{\delta,1,h^p_\varepsilon}(\varpi_p^{-1}(B_{\Gamma_p}(4)))}
\right).
\end{align*}
By observing the behavior of weighted H\"older norms under the rescaling, we obtain 
\begin{align*}
\| u\|_{C^{k+2,\alpha}_{\delta,\varepsilon,g_\varepsilon}(Z_{3\varepsilon})}
&\le C_2\left( \| \Delta_{g_\varepsilon} u\|_{C^{k,\alpha}_{\delta-2,\varepsilon,g_\varepsilon}(Z_{4\varepsilon})}
+\| u\|_{C^0_{\delta,\varepsilon,g_\varepsilon}(Z_{4\varepsilon})}
\right).
\end{align*}

Next, we consider the estimate on $Y_{p,t_0,2\varepsilon}$. The flat Riemannian manifold  $(Y_{p,t_0,2\varepsilon},g_0)$ can be isometrically embedded in the Riemannian cones $(\C^2\setminus\{ 0\})/\Gamma_p\cong S^3/\Gamma_p\times \R_+$. 
For $\lambda>0$, define $T_\lambda\colon (\C^2\setminus\{ 0\})/\Gamma_p\to (\C^2\setminus\{ 0\})/\Gamma_p$ by $T_\lambda(x)=\lambda\cdot x$.

We put $Y'_n:=Y_{p,2^{-n}t_0,2^{-n-2}t_0}$ and $Y_n:=Y_{p,2^{-n+1}t_0,2^{-n-3}t_0}$. 
Then $Y'_n\subset Y_n$ and $Y_{p,t_0,2\varepsilon}\subset\bigcup_{n=1}^NY'_n$ for some $N$. Here, $N$ is the smallest positive integer such that $2^{-N-2}t_0\le 2\varepsilon$. 
Since $2^{2n}T_{2^{-n}}^*g_0=g_0$, if we put $g_{n,\varepsilon}:=2^{2n}T_{2^{-n}}^*(g_\varepsilon|_{Y_n})$, 
the family of metrics $\{ g_{n,\varepsilon}\}_{n,0<\varepsilon\ll 1}$ 
is included in $\mathfrak{M}(g_0,Y_0,(A_k)_k)$ for some constants $A_k$ by \eqref{eq asymp2}. Therefore, there is a constant $C_3>0$ independent of $\varepsilon,n$ such that 
\begin{align*}
\| u\|_{C^{k+2,\alpha}_{\delta,1,g_{n,\varepsilon}}(Y'_0)}
&\le C_3\left( 2^{-2n}\| \Delta_{T_{2^{-2n}}^*(g_\varepsilon|_{Y_n})} u\|_{C^{k,\alpha}_{\delta-2,1,g_{n,\varepsilon}}(Y_0)}
+\| u\|_{C^0_{\delta,1,g_{n,\varepsilon}}(Y_0)}
\right).
\end{align*}
By the behavior of the weighted H\"older norm under the rescaling, and patching the estimates for all $n=0,1,\ldots$, there is a constant $C_3'>0$ independent of $\varepsilon$ such that  
\begin{align*}
\| u\|_{C^{k+2,\alpha}_{\delta,r,g_\varepsilon}(Y_{p,t_0,2\varepsilon})}
&\le C_3'\left( \| \Delta_{g_\varepsilon} u\|_{C^{k,\alpha}_{\delta-2,r,g_\varepsilon}(Y_{p,2t_0,\varepsilon})}
+\| u\|_{C^0_{\delta,r,g_\varepsilon}(Y_{p,2t_0,\varepsilon})}
\right).
\end{align*}
By patching the estimates on $X_{t_0/2}, Y_{p,t_0,2\varepsilon},Z_{3\varepsilon}$, we have the assertion. 
\end{proof}

Theorem \ref{thm schauder} follows directly from the next proposition. 
\begin{prop}
Let $g_\varepsilon$ be Riemannian metrics on $X=X(\Lambda,\Gamma)$ satisfying \eqref{eq asymp1}\eqref{eq asymp2}. Denote by $\mu_{g_\varepsilon}$ the Riemannian measure of $g_\varepsilon$ and put 
\begin{align*}
E^{k,\alpha}_{\delta,g_\varepsilon}
:=\left\{ \psi\in C^{k,\alpha}_{\delta,g_\varepsilon}\left|\, \int_X\psi d\mu_{g_\varepsilon}=0\right.\right\}.
\end{align*}
Let $-2<\delta<0$. 
Then there are positive constants $\varepsilon_0$ and $K_{k,\delta}$ independent of $\varepsilon$ such that 
\begin{align*}
\| u\|_{C^{k+2,\alpha}_{\delta,g_\varepsilon}}
\le K_{k,\delta}\| \Delta_{g_\varepsilon}(u)\|_{C^{k,\alpha}_{\delta-2,g_\varepsilon}}
\end{align*}
for all $0<\varepsilon\le\varepsilon_0$ and $u\in E^{k+2,\alpha}_{\delta,g_\varepsilon}(X)$. 
\label{prop weight sch int 0}
\end{prop}

\begin{proof}
We prove by contradiction. 
If we deny the conclusion, then by Proposition \ref{prop weight sch}, 
there are $\varepsilon_n>0$ and $u_n\in E^{k+2,\alpha}_{\delta,g_{\varepsilon_n}}(X)$ such that 
$\| u_n\|_{C^0_{\delta,g_{\varepsilon_n}}}=1$ and $\lim_{n\to 0}\| \Delta_{g_{\varepsilon_n}}(u_n)\|_{C^{k,\alpha}_{\delta-2,g_{\varepsilon_n}}}=0$. 
If $\inf_n\varepsilon_n>0$, it contradicts  the injectivity of $\Delta_{g_{\varepsilon}}$, with $\varepsilon=\inf_n\varepsilon_n$. 
We may suppose $\varepsilon_n\searrow 0$. In this argument, if we can derive a convergent subsequence, we can then assume that the original sequence itself converges by replacing the sequence with its convergent subsequence, for the simplicity of the argument. 

By the assumption, $\| u_n\|_{C^{k+2,\alpha}_{\delta,g_{\varepsilon_n}}}$ is bounded. On every compact subset $D$ of $X\setminus \pi^{-1}(S)$, $\sigma_{\varepsilon_n}$ are bounded from below by a constant independent of $n$. Therefore, $\| u_n|_D\|_{C^{k+2,\alpha}_{g_{\varepsilon_n}}}$ is bounded and $\{ \nabla_{g_{\varepsilon_n}}^ju_n|_D\}_n$ are uniformly bounded. Moreover, since $R\sigma_{\varepsilon_n}$ are also bounded from below, there is a constant $\varepsilon'>0$ such that $d_{g_{\varepsilon_n}}(x,y)<\varepsilon'$ implies $x\sim y$ for all $x,y\in D$ and $n$. Therefore, $\{ \nabla_{g_{\varepsilon_n}}^ju_n|_D\}_n$ are equicontinuous. Then by Ascoli-Arzel\`a Theorem, we may suppose that there is $u\in C^{k+2}(X\setminus \pi^{-1}(S))$ such that $u_n\to u$ in $C^{k+2}(D)$ on every compact subset $D$ of $X\setminus \pi^{-1}(S)$. On $D$, $g_{\varepsilon_n}\to \pi^*g_0$ with respect to every $C^k$-norm, hence $\Delta_{\pi^*g_0}(u)=0$ on $X\setminus \pi^{-1}(S)$. Now, we regard $u$ as a function on $\T\setminus S$, then $\Delta_{g_0}(u)=0$. By the uniform bound of $u_n$ with respect to $C^0_{\delta,g_{\varepsilon_n}}$-norm, we have 
$|u|\le Cd_\T(\cdot,p)^{-1}$ on a small neighborhood of $p\in S$. Since $u$ is harmonic with respect to the flat metric $g_0$, every $p$ is a removable singularity and $u$ can be extended to a smooth function on the orbifold $\T$. If we denote by $F\colon \T^4_\Lambda\to \T$ the quotient map, then $F^*u$ is a harmonic function on $4$-torus, hence it is constant. By the condition that $\int_X u_n d\mu_{g_{\varepsilon_n}}=0$, we can see that $u\equiv 0$. 

By $\| u_n\|_{C^0_{\delta,g_{\varepsilon_n}}}=1$, 
there are $p_n\in X$ such that 
\begin{align}
\sigma_{\varepsilon_n}(p_n)^{-\delta}|u_n(p_n)|=1\label{eq max un}
\end{align}
for every $n$. Since we have shown $u\equiv 0$, we can see that there is no subsequence of $(u_n)_n$ whose limit is in $X\setminus \pi^{-1}(S)$. 
Now, by taking a subsequence, we may suppose $\lim_{n\to \infty}d_\T(\pi(p_n),p)=0$ for some $p\in S$. Put $r_n:=d_\T(\pi(p_n),p)$. 

Assume $\sup_nr_n/\varepsilon_n=\infty$. 
We consider $v_n:=r_n^{-\delta}T_{r_n}^*(u_n|_{Y_{p,t_0,\varepsilon_n}})$. 
Since $\| u_n\|_{C^{k+2,\alpha}_{\delta,g_{\varepsilon_n}}}$ is bounded, we can see that $\| v_n\|_{C^{k+2,\alpha}_{\delta,r_n^{-2}T_{r_n}^*g_{\varepsilon_n}}}$ is bounded and $r_n^{-2}T_{r_n}^*g_{\varepsilon_n}\to g_0$ with respect to the $C^k$-norm on every compact subset of $\C^2\setminus\{ 0\}/\Gamma_p$. 
Therefore, there is a $C^{k+2}$-function $v$ on $\C^2\setminus\{ 0\}/\Gamma_p$ such that $v_n\to v$ in $C^{k+2}$-topology on every compact subset in $\C^2\setminus\{ 0\}/\Gamma_p$. Moreover, we can show that $v$ is harmonic with respect to $g_0$ and $|v|=O(r^\delta)$. 
However, since we have assumed $-2<\delta<0$, such harmonic functions are always $0$, which contradicts \eqref{eq max un}. 

Accordingly, we may suppose $\sup_nr_n/\varepsilon_n<\infty$. By taking a subsequence, we may assume that the sequence $(\mathcal{I}_{p,\varepsilon_n}(p_n))_n$ converges to some $q\in M_p$. Similarly to the above argument, if we put $v_n':=\varepsilon_n^{-\delta}(\mathcal{I}_{p,\varepsilon_n}^{-1})^*(u_n|_{U(p,t_0)})$, then $v'_n$ converges to some $v'\in C^{k+2}(M_p)$ on every compact set in $M_p$, $v'$ is harmonic with respect to $g_p$ and $|v'|=O(r^\delta)$. 
By the maximum principle for the harmonic functions on $M_p$, we have $v'=0$, which contradicts \eqref{eq max un}. 
\end{proof}

\subsection{Calabi-Yau structures on $X$}
Let $X,\Omega,\pi,S,\omega_0,\Omega_0$ be as in Subsection \ref{subsec kummer} and assume $(M_p,\eta_p,\Theta_p,\varpi_p)_{p\in S}\approx (\T,S)$. Let $g_\varepsilon$ be the K\"ahler metric of $\omega_\varepsilon$ constructed in Subsection \ref{subsec appro}.
\begin{definition}
\normalfont
Let $T_\varepsilon\in\Omega^2(X)$ for $\varepsilon>0$, $T_0\in \Omega^2(\T\setminus S)$ and $T_p\in\Omega^2(M_p)$. 
We write $T_\varepsilon\stackrel{\T\setminus S}{\longrightarrow} T_0$ as $\varepsilon\to 0$ if 
\begin{align*}
\lim_{\varepsilon\to 0}\sup_K|\nabla_0^k\{ (\pi^{-1})^*T_\varepsilon -T_0\}|_{g_0}=0
\end{align*}
for any $k\ge 0$ and compact subset $K\subset \T\setminus S$. We write $T_\varepsilon\stackrel{S}{\to} (T_p)_{p\in S}$ as $\varepsilon\to 0$ if 
\begin{align*}
\lim_{\varepsilon\to 0}\sup_K|\nabla_{g_p}^k\{ (\mathcal{I}_{p,\varepsilon}^{-1})^*T_\varepsilon -T_p\}|_{g_p}=0
\end{align*}
for any $k\ge 0$, $p\in S$ and compact subsets $K\subset M_p$. Here, $g_0,g_p$ are K\"ahler metrics of $\omega_0,\eta_p$, respectively. 
\label{def conv bubble}
\end{definition}

\begin{lem}\label{lem conv 2-form}
Let $T_\varepsilon,\tilde{T}_\varepsilon\in\Omega^2(X)$ for $\varepsilon>0$, $T_0\in \Omega^2(\T\setminus S)$ and $T_p\in\Omega^2(M_p)$. Assume that 
\begin{align*}
\| \tilde{T}_\varepsilon-T_\varepsilon\|_{C^{k,\alpha}_{\delta-2,g_\varepsilon}}&\lesssim \varepsilon^{\delta'}
\end{align*}
for all $k\ge 0$ and some $\delta'\in\R$. 
\begin{itemize}
\setlength{\parskip}{0cm}
\setlength{\itemsep}{0cm}
 \item[$({\rm i})$] If $T_\varepsilon|_{X\setminus\bigcup \overline{U(p,2\varepsilon^{1/2})}}=\pi^*T_0|_{X\setminus\bigcup\overline{U(p,2\varepsilon^{1/2})}}$ and $\delta'>0$, then $\tilde{T}_\varepsilon\stackrel{\T\setminus S}{\rightarrow} T_0$.
 \item[$({\rm ii})$] If $T_\varepsilon|_{U(p,\varepsilon^{1/2})}=\varepsilon^2\mathcal{I}_{p,\varepsilon}^*T_p|_{U(p,\varepsilon^{1/2})}$ and $\delta'>2-\delta$, then $\varepsilon^{-2}\tilde{T}_\varepsilon\stackrel{S}{\rightarrow} (T_p)_{p\in S}$. 
\end{itemize}
\end{lem}
\begin{proof}
Let $K\in \T\setminus S$, $K_p\subset M_p$ be compact subsets. 
By the assumption, we have $\sup_K|\nabla_0^k((\pi^{-1})^*\tilde{T}_\varepsilon-T_\varepsilon)|_{g_0}\lesssim \varepsilon^{\delta'}$. Therefore, we have $T_\varepsilon\stackrel{\T\setminus S}{\rightarrow} T_0$. By $g_\varepsilon|_{U(p,\varepsilon^{1/2})}=\varepsilon^2\mathcal{I}_{p,\varepsilon}^*g_p$ and by the assumptions, we have 
\begin{align*}
\sup_{K_p}\left|\nabla_{g_p}^k(\varepsilon^{-2}(\mathcal{I}_{p,\varepsilon}^{-1})^*\tilde{T}_\varepsilon-T_p)\right|_{g_p}
&=\sup_{\mathcal{I}_{p,\varepsilon}^{-1}(K_p)}\varepsilon^k\left|\nabla_{g_\varepsilon}^k\{(\tilde{T}_\varepsilon-T_\varepsilon)\}\right|_{g_\varepsilon}\\
&\le\varepsilon^{\delta-2}\sup_{\mathcal{I}_{p,\varepsilon}^{-1}(K_p)}\sigma_\varepsilon^{-\delta+2+k}\left|\nabla_{g_\varepsilon}^k\{(\tilde{T}_\varepsilon-T_\varepsilon)\}\right|_{g_\varepsilon}\\
&\le \varepsilon^{\delta-2}\| \tilde{T}_\varepsilon-T_\varepsilon\|_{C^{k,\alpha}_{\delta-2,g_\varepsilon}}
\lesssim \varepsilon^{\delta+\delta'-2}.
\end{align*}
It implies $\varepsilon^{-2}\tilde{T}_\varepsilon\stackrel{S}{\rightarrow} (T_p)_{p\in S}$. 
\end{proof}

As a consequence, we obtain the next result. 
\begin{thm}
Let $X=X(\Lambda,\Gamma)$ and $\Omega,\pi,S,\omega_0,\Omega_0$ be as in Subsection \ref{subsec kummer}. For every $p\in S$, we assume that an ALE Calabi-Yau manifold $(M_p,\eta_p,\Theta_p,\varpi_p)$ asymptotic to $\C^2/\Gamma_p$ is given, where $\Gamma_p\subset \Gamma$ is the stabilizer of $p$. For any $0<\varepsilon \ll 1$, there is a K\"ahler form $\tilde{\omega}_\varepsilon\in \Omega^{1,1}(X)$ such that 
\begin{align*}
\tilde{\omega}_\varepsilon^2&=\frac{1}{2}\Omega\wedge\overline{\Omega},\quad
\tilde{\omega}_\varepsilon\stackrel{\T\setminus S}{\longrightarrow} \omega_0,\quad\varepsilon^{-2}\tilde{\omega}_\varepsilon\stackrel{S}{\to} (\eta_p)_{p\in S}
\end{align*}
\label{thm RFK and appro}
\end{thm}
\begin{proof}
Fix $-2<\delta<0$. We obtain $\varphi_\varepsilon\in C^\infty(X)$ by Theorem \ref{thm fix pt} and put 
$\tilde{\omega}_\varepsilon:=W_\varepsilon^{-1/2}(\omega_\varepsilon+\sqrt{-1}\ddb\varphi_\varepsilon)$. Since $\| \varphi_\varepsilon\|_{C^{k+2,\alpha}_{\delta,g_\varepsilon}}\lesssim\varepsilon^{3-\delta/2}$ for all $k\ge 0$ and $|W_\varepsilon-1| \lesssim \varepsilon^4$, we have $\|\tilde{\omega}_\varepsilon-\omega_\varepsilon\|_{C^{k,\alpha}_{\delta-2,g_\varepsilon}}\lesssim\varepsilon^{3-\delta/2}$. Therefore, the assertion follows from Lemma \ref{lem conv 2-form}, since both of $3-\delta/2$ and $1+\delta/2$ are positive. 
\end{proof}

\subsection{The Levi-Civita connections}\label{subsec estimate on LC}
Now, we have the K\"ahler metric $g_\varepsilon$ of $\omega_\varepsilon$ and the Ricci-flat K\"ahler metric $\tilde{g}_\varepsilon$ of $\tilde{\omega}_\varepsilon$ by Theorem \ref{thm RFK and appro}. 
In the above sections, we often use the Levi-Civita connection of $g_\varepsilon$ to consider the weighted H\"older norms. However, in the following sections, we also consider the derivation by the Levi-Civita connection $\tilde{\nabla}$ of $\tilde{g}_\varepsilon$. Here, we consider some estimates on the higher derivatives by $\tilde{g}_\varepsilon$. 

Let $g,\tilde{g}$ be Riemannian metrics on $X$ and denote by $\nabla,\tilde{\nabla}$ their Levi-Civita connections, respectively. We put $\tilde{\nabla}=\nabla+A$ for some $A\in\Gamma(T^*X\otimes {\rm End}(TX))$. 
We also assume 
\begin{align*}
\frac{g}{2}\le\tilde{g}\le 2g,
\quad |\nabla^k\tilde{g}|_g\le s_k
\end{align*}
for some nonnegative valued functions $\{ s_k\}_{k=1}^\infty$. If we denote by $\Psi\colon T^*X\otimes T^*X\otimes T^*X\to T^*X\otimes T^*X\otimes T^*X$ the bundle map defined by 
\begin{align*}
\Psi(dx^i\otimes dx^j\otimes dx^k)
&=dx^i\otimes dx^j\otimes dx^k+dx^j\otimes dx^i\otimes dx^k\\
&\quad\quad -dx^k\otimes dx^i\otimes dx^j
\end{align*}
for a local coordinate $x^1,\ldots,x^4$, then we may write 
\begin{align*}
A=\tilde{g}^{-1}\Psi(\nabla \tilde{g}).
\end{align*}
Since $\Psi$ is parallel with respect to any Riemannian metric, then we have 
\begin{align*}
|\nabla^kA|_g\le \sum_{\substack {m_1+\cdots +m_l=k+1 \\ m_i>0}} C_{m_1,\ldots,m_l}s_{m_1}\cdots s_{m_l}
\end{align*}
for some positive constants $C_{m_1,\ldots,m_l}$. Moreover, we can see 
\begin{align*}
\tilde{\nabla}^k&=(\nabla+A)^k\\
&=\nabla^k+\sum_{\substack{m_1+\cdots +m_l=k-l-j \\ l\ge 1,\, m_i\ge 0}} C'_{j,m_1,\ldots,m_l}(\nabla^{m_l}A)(\nabla^{m_{l-1}}A)\cdots(\nabla^{m_1}A)\nabla^j
\end{align*}
for some constants $C'_{m_0,\ldots,m_l}$. 
Therefore, for any tensor $T$ on $X$, we have 
\begin{align}
|\tilde{\nabla}^kT|_{\tilde{g}}
&\le C'|\nabla^kT|_g
+\sum_{j=0}^{k-1}\sum_{\substack{m_1+\cdots +m_d=k-j\\ m_i>0}} C'_{j,m_1,\ldots,m_d}
s_{m_1}\cdots s_{m_d}|\nabla^jT|_g\label{ineq nabla k 1}
\end{align}
for some positive constants $C',C'_{j,m_1,\ldots,m_d}$. 

Now, we put $\tilde{g}=\tilde{g}_\varepsilon$ and $g=g_\varepsilon$. We have already shown that 
\begin{align*}
\tilde{\omega}_\varepsilon-\omega_\varepsilon&=(W_\varepsilon^{-1/2}-1)\omega_\varepsilon+\sqrt{-1}W_\varepsilon^{-1/2}\ddb\varphi_\varepsilon,\\
\nabla_{g_\varepsilon}^k\tilde{\omega}_\varepsilon
&=\sqrt{-1}W_\varepsilon^{-1/2}\nabla_{g_\varepsilon}^k\ddb\varphi_\varepsilon,\\
\|\varphi_\varepsilon\|_{C^{k+2,\alpha}_{\delta,g_\varepsilon}}&\lesssim\varepsilon^{3-\delta/2}.
\end{align*}
Here, the third estimate implies $|\nabla^k\ddb\varphi_\varepsilon|\lesssim  \varepsilon^{3-\delta/2}\sigma_\varepsilon^{\delta-k-2}$ for all $k$, then the first and the second equalities implies that $g_\varepsilon/2\le \tilde{g}_\varepsilon\le 2g_\varepsilon$ and $s_k\lesssim \varepsilon^{3-\delta/2}\sigma_\varepsilon^{\delta-k-2}$. Therefore, we obtain 
\begin{align*}
|\tilde{\nabla}^kT|_{\tilde{g}_\varepsilon}
&\lesssim |\nabla^kT|_{g_\varepsilon}
+\sum_{j=0}^{k-1}\sum_{d=1}^{k-j} \varepsilon^{3d-\delta d/2}\sigma_\varepsilon^{\delta d-2d-k+j}|\nabla^jT|_{g_\varepsilon}.
\end{align*}
Recall $-2<\delta<0$. Then we have $\sigma_\varepsilon^{\delta d-2d}\le \varepsilon^{\delta d-2d}$ and 
$\varepsilon^{3d-\delta d/2}\sigma^{\delta d-2d} \le \varepsilon^{1+\delta /2}$. Thus, we can see 
\begin{align}
\sigma_\varepsilon^{-\delta'+k}|\tilde{\nabla}^kT|_{\tilde{g}_\varepsilon}
&\lesssim \sigma_\varepsilon^{-\delta'+k}|\nabla^kT|_{g_\varepsilon}
+\sum_{j=0}^{k-1}\varepsilon^{1+\delta /2}\sigma_\varepsilon^{-\delta'+j}|\nabla^jT|_{g_\varepsilon}\label{ineq nabla k}
\end{align}
for any $\delta'\in\R$. Consequently, we obtain the next proposition. 
\begin{prop}\label{prop Ck comparison}
There is a constant $C_k,\varepsilon_0>0$ independent of $\varepsilon,\delta,\delta'$ such that 
\begin{align*}
\| T\|_{C^k_{\delta',\tilde{g}_\varepsilon}}
&\le C_k\| T\|_{C^k_{\delta',g_\varepsilon}}
\end{align*}
for all $T\in C^k(X)$, $\delta\in(-2,0)$, $\delta'\in \R$, and $\varepsilon\in (0,\varepsilon_0]$. 
\end{prop}

\section{Anti-self-dual $2$-forms on the ALE spaces}\label{sec ASD ALE}
On an oriented Riemannian manifold $(M,g)$ of dimension $4$, $2$-form $\alpha\in\Omega^2(M)$ is called {\it self-dual} if $*\alpha=\alpha$ and 
{\it anti-self-dual} if $*\alpha=-\alpha$. It is easy to see that self-dual or anti-self-dual $2$-forms are harmonic iff it is closed. 

In this section, let $(M,\eta,\Theta,\varpi)$ be an ALE Calabi-Yau manifold asymptotic to $\C^2/\Gamma$. 

\subsection{The Chern forms of ASD connections}\label{subsec instanton ALE}
On every ALE space, Kronheimer and Nakajima obtained the following result in  \cite{KN1990instanton}. 
\begin{thm}[{\cite{KN1990instanton}}]\label{thm ASD on ALE}
Let $N_\Gamma$ be the number of the conjugacy classes of unitary representations of $\Gamma$. 
For $i=1,\ldots,N_\Gamma$, there are hermitian vector bundles $\mathcal{R}_i$ with hermitian connections $A_i$ such that the following holds. 
\begin{itemize}
\setlength{\parskip}{0cm}
\setlength{\itemsep}{0cm}
 \item[$({\rm i})$] For the curvature form $F^{A_i}\in \Omega^2({\rm End}(\mathcal{R}_i))$ we have $\int_M|F^{A_i}|^2\eta^2<\infty$. 
 \item[$({\rm ii})$] The first Chern classes $c_1(\mathcal{R}_i)$ form a basis for $H^2(M)$. 
\end{itemize}
\end{thm}
We review of the construction of $(\mathcal{R}_i,A_i)$ briefly. 
Let $\hat{M}$, $G$, $\mu$ be as in Proposition  \ref{prop hk quotient}. We have already obtained $M$ as the quotient space $\mu_1^{-1}(\zeta)\cap\mu_\C^{-1}(0)/G\cong \mu_\C^{-1}(0)_\zeta/G_\C$ for some $\zeta\in\mathfrak{g}^*$. 
Here, $P_\zeta:=\mu_1^{-1}(\zeta)\cap\mu_\C^{-1}(0)\to M$ is a principal $G$-bundle. By the induced Riemannian metric on $P_\zeta$, we can take the horizontal distribution of $TP_\zeta$ as the orthogonal complement of the $G$-orbits. It defines a $G$-connection $A_\zeta$ on $P_\zeta$. 

Let $R_i$ be the irreducible unitary representations of $\Gamma$, where $i=0,\ldots N_\Gamma$. We suppose that $R_0\cong\C$ is the trivial representation. 
Then $G$ is isomorphic to $\prod_{i\neq 0}U(R_i)$. Denote by $p_i\colon G\to U(R_i)$ the $i$-th projection, then the associate hermitian bundles $\mathcal{R}_i:=P_\zeta\times_{p_i}R_i$ and its hermitian connections $A_i$ are induced by $A_\zeta$ for every $i=1,\ldots,N_\Gamma$.

Let ${\bf H}$ be the $\C^\times$-action defined in Proposition \ref{prop hk quotient}. If we consider the subgroup $S^1:=\{ \lambda\in\C^\times|\, |\lambda|=1\}$, then the $S^1$-actions is commutative with $G$-actions and preserving the metric on $\hat{M}$, hence $A_\zeta$ is also preserved by the $S^1$-action.

\begin{prop}\label{prop instanton}
Put 
\begin{align*}
c_1(A_i):=\frac{\sqrt{-1}}{2\pi}{\rm tr}(F^{A_i})\in \Omega^{1,1}(M).
\end{align*}
\begin{itemize}
\setlength{\parskip}{0cm}
\setlength{\itemsep}{0cm}
 \item[$({\rm i})$] $c_1(A_i)$ is an anti-self-dual $2$-form and $\|c_1(A_i)\|_{L^2}<\infty$. 
 \item[$({\rm ii})$] $c_1(A_1),\ldots,c_1(A_{N_\Gamma})$ are linearly independent.
 \item[$({\rm iii})$] ${\bf H}_\lambda^*c_1(A_i)=c_1(A_i)$ for all $i$ and $\lambda\in S^1$. 
  \item[$({\rm iv})$] There is $\gamma_i\in\Omega^1(M')$ on $M'=M\setminus\varpi^{-1}(0)$ such that ${\bf H}_{\lambda}^*\gamma_i=\gamma_i$, $c_1(A_i)=d\gamma_i$ and $|\nabla_0^k\gamma_i|_{g_0}=O(r^{-3-k})$ for all $i$ and $k\ge 0$. Here, $g_0$ is the standard metric on $ (\C^2\setminus\{ 0\})/\Gamma$, $\nabla_0$ is its Levi-Civita connection, and we identify $M'$ with $(\C^2\setminus\{ 0\})/\Gamma$ by $\varpi$. 
\end{itemize}
\end{prop}
\begin{proof}
By \cite{GochoNakajima}, $A_\zeta$ is an anti-self-dual connection, so are $A_i$. $A_i$ is $L^2$-finite by $({\rm i})$ of Theorem \ref{thm ASD on ALE}. The linear independence of $c_1(A_1),\ldots,c_1(A_{N_\Gamma})$ follows from $({\rm ii})$ of Theorem \ref{thm ASD on ALE}. 
Since $A_\zeta$ is preserved by the $S^1$-action we have ${\bf H}_\lambda^*c_1(A_i)=c_1(A_i)$.

Next, we consider the conformally compactification $(\overline{M}=M\cup\{\infty\},g'=\phi^2 g)$, where $\varphi$ is a positive valued function of $r$ such that $\phi=1/r^2$ on some neighborhood of $\infty$. Then $\overline{M}$ is a compact orbifold and $g'$ is $C^{3,\alpha}$-metric for $0<\alpha<1$. On the neighborhood of $\infty$, we fix a harmonic coordinate of $g'$. Then by the elliptic regularity theory discussed in the proof of \cite[Theorem 6.4]{TV2005adv}, $g'$ is a smooth orbifold metric on $\overline{M}$. Now, $A_i$ is still anti-self-dual and has finite action with respect to $g'$. Then by the Removable Singularity Theorem for Yang-Mills connections \cite{Uhlenbeck1982}, $A_i$ can be extended to the smooth connection at $\infty$. Take an orthonormal frame $E_1,\ldots,E_{\dim R_i}$ of $\mathcal{R}_i$ around $\infty$ such that $\nabla^{A_i}E_\alpha|_\infty=0$. Here, we can take the frame globally as the pullback by $\C^2\setminus\{ 0\}\to \C^2\setminus\{ 0\}/\Gamma$.  Then the connection form $\hat{\gamma}_i\in\Omega^1(\C^2\setminus\{ 0\})\otimes \mathfrak{u}(R_i)$ of $A_i$ satisfies $|\nabla_0^k(\hat{\gamma}_i)|=O(r^{-3-k})$ for all $k\ge 0$. 
Now, we have $F^{A_i}=d\hat{\gamma}_i+[\hat{\gamma}_i,\hat{\gamma}_i]$ with respect to the above trivialization. If we put $\gamma_i:=\frac{\sqrt{-1}}{2\pi}{\rm tr}\hat{\gamma}_i$, then it descends to a $1$-form on $M'$ and we have 
$c_1(A_i)=d\gamma_i$. By the asymptotic behavior  for $\hat{\gamma}_i$, we also have 
$|\nabla_0^k(\gamma_i)|=O(r^{-3-k})$. 
Here, $\gamma_i$ is not invariant under the $S^1$-action, however, we can replace it by 
$\frac{1}{2\pi}\int_0^{2\pi}({\bf H}_{e^{\sqrt{-1}t}}^*\gamma_i)dt$, we may suppose $\gamma_i$ is $S^1$-invariant.
\end{proof}

Put $\V=\C^2\setminus\{ 0\}/\Gamma$ and identify with $M'$ by $\varpi$. The $S^1$-action on $\V$ identified with ${\bf H}$ is the multiplication of the scalar. 
\begin{prop}
Let $\gamma\in \Omega^1(\V)$ be $S^1$-invariant and $d\gamma\in\Omega^{1,1}(\V)$. We suppose $|\nabla_0^k\gamma|_{g_0}=O(r^{-3-k})$. Then there is an $S^1$-invariant $\psi=\psi^\gamma\in C^\infty(\V)$ such that $d\gamma=\sqrt{-1}\ddb \psi$ and 
$|\nabla_0^k\psi|_{g_0}=O(r^{-2-k})$ for all $k\ge 0$. Here, $I$ is the standard complex structure on $\V$. 
\label{prop instanton potential}
\end{prop}
\begin{proof}
Suppose the assertion has been proved for $\Gamma=\{ 1\}$. For general $\Gamma$, we can regard $\gamma$ as a $\Gamma$-invariant $1$-form on $\C^2\setminus\{ 0\}$, then we have $S^1$-invariant $\psi$ with $F=d\psi\circ I$ and 
$|\nabla_0^k\psi|_{g_0}=O(r^{-2-k})$ for all $k\ge 0$. Since $\Gamma$ is finite, we can replace $\psi$ by $\frac{1}{\#\Gamma}\sum_{\gamma\in\Gamma}\gamma^*\psi$, which is $\Gamma$-invariant. Consequently, it suffices to show the case of $\Gamma=\{ 1\}$. 

Now, we reformulate the assumptions on $\gamma$ in terms of Sasakian geometry.
Denote by $\xi\in \mathcal{X}(S^3)$ the vector field defined by $\xi_p:=\frac{d}{dt}e^{\sqrt{-1}t}p$ and $\sigma:=g_{S^3}(\xi,\colon)\in\Omega^1(S^3)$. Here, $g_{S^3}$ is the standard metric on the unit sphere. Then $(S^3,g_{S^3},\sigma,\xi)$ is a Sasakian manifold such that $\xi$ is the Reeb vector field and $\sigma$ is the contact form. We define a distribution $D$ on $S^3$ by $D:={\rm Ker}(\sigma)\subset TS^3$. We identify $\V$ with $S^3\times \R_+$ by the polar coordinate, then we have the decomposition $T_{(r,p)}^*\V=\R dr\oplus \R\sigma_p\oplus D^*_p$, where $D^*_p=\{ \alpha\in T^*_pS^3|\, \alpha(\xi)=0\}$. We put $\gamma=fdr+h\sigma+\tau$ for some $f,h\in C^\infty(\V)$ and $\tau=\tau(r,\cdot)\in \Gamma(D^*)$. By the $S^1$-invariance of $\gamma$, we have $\xi(f)=\xi(h)=\iota_\xi d\tau=0$. $d\gamma$ is of type $(1,1)$ with respect to $I$ iff 
\begin{align}
\frac{\del\tau}{\del r}
=d_Bf-\frac{d_Bh\circ I}{r},\quad d_B\tau\circ I=d_B\tau.\label{eq type 1,1}
\end{align}
Here, $d_B=d-\sigma\wedge \iota_\xi d$. Note that $d_Bf,d_Bh\in\Gamma(D^*)$, $d_B\tau\in\Gamma(\Lambda^2D^*)$ and $I$ preserves $D^*$. By $|\gamma|_{g_0}=O(r^{-3})$ and $g_0=dr^2+r^2g_{S^3}$, we have $|f|=O(r^{-3})$, $|h/r|=O(r^{-3})$ and $|\tau|_{g_{S^3}}=O(r^{-2})$. 
By the first equality of \eqref{eq type 1,1}, we have 
\begin{align*}
\tau=d_B\hat{f}+d_B\hat{h}\circ I,
\end{align*}
where 
\begin{align*}
\hat{f}:=-\int_r^\infty f(t,\cdot)dt,\quad
\hat{h}:=\int_r^\infty \frac{h(t,\cdot)}{t}dt.
\end{align*}
We can check that $\hat{f},\hat{h}$ are $S^1$-invariant. 
Then the second equality of \eqref{eq type 1,1} is also satisfied since we have $d_B^2\hat{f}=0$ and $d\sigma,d_B(d_B\hat{h}\circ I)$ is of type $(1,1)$. Moreover, we may write $\gamma=d\hat{f}+d\hat{h}\circ I$, $d\gamma=-2\sqrt{-1}\ddb\hat{h}$. Here, we have $|\hat{h}|=O(r^{-2})$. 
We put $\psi=-2\hat{h}$ and estimate the higher derivatives of $\hat{h}$. 
For $x=(x_1,\ldots,x_4)\in \V=\R^4\setminus\{ 0\}$, we have 
\begin{align*}
\hat{h}(x)=\int_{\| x\|}^\infty\frac{h(tx/\| x\|)}{t}dt.
\end{align*}
For a multi-index $J=(j_1,\ldots,j_l)\in \{ 1,\ldots,4\}^l$, we put $\del_J=\frac{\del^k}{\del x_{j_1}\cdots\del x_{j_l}}$ and $l=|J|$. 
By the induction, for a multi-index $K$ with $k=|K|$, we may write 
\begin{align*}
\del_K\hat{h}
&=\sum_{j+|J|= k,|J|\le k-1}\sum_{m\ge 0}\frac{p_{j,m,J}(x)}{\| x\|^{j+m}}\del_Jh(x)\\
&\quad\quad
+\sum_{j+|J|\le k,|J|\ge 1}\sum_{m\ge 0}\frac{q_{j,m,J}(x)}{\| x\|^{j+|J|+m}}\int_{\| x\|}^\infty t^{|J|-1}\del_Jh\left(\frac{tx}{\| x\|}\right)dt,
\end{align*}
for some polynomial $p_{j,m,J},q_{j,m,J}$ whose degrees are $m$. Since $h=\iota_\xi \gamma$ and $|\nabla_0^k\xi|_{g_0}=O(r^{1-k})$, we have $|\del_J h|=O(r^{-2-|J|})$. Therefore, we have $|\del_K \hat{h}|=O(r^{-2-|K|})$. 
\end{proof}

\subsection{Closed $(1,1)$-forms on $X$ coming from the bubbles}
We construct closed $(1,1)$-forms on  $X=X(\Lambda,\Gamma)$ from the harmonic $2$-forms obtained in Subsection \ref{subsec instanton ALE}. Let $(M_p,\eta_p,\Theta_p,\varpi_p)_{p\in S}\approx (\T,S)$ and $A_{p,1},\ldots, A_{p,N_{\Gamma_p}}$, $\gamma_{p,1},\ldots,\gamma_{p,N_{\Gamma_p}}$ be the anti-self-dual connections and $1$-forms on $M_p$ obtained by Proposition \ref{prop instanton}. Denote by $\psi_{p,i}=\psi^{\gamma_{p,i}}\in C^\infty(\C^2\setminus\{ 0\}/\Gamma_p)$ the functions obtained by Proposition \ref{prop instanton potential}. Now, we define $\Xi'_{p,i,\varepsilon}\in\Omega^{1,1}(X)$ by 
\begin{align*}
\Xi'_{p,i,\varepsilon}|_{\overline{U(p,\varepsilon^{1/2})}}&:=\varepsilon^2\mathcal{I}_{p,\varepsilon}^*c_1(A_{p,i}),\\
\Xi'_{p,i,\varepsilon}|_{A(p,\varepsilon^{1/2})}&:=\sqrt{-1}\ddb \left( \chi_{\varepsilon^{1/2}}\cdot \varepsilon^2\mathcal{I}_{p,\varepsilon}^*\varpi_p^*\psi_{p,i}\right),\\
\Xi'_{p,i,\varepsilon}|_{X\setminus\overline{U(p,2\varepsilon^{1/2})}}&:=0,
\end{align*}
which are closed by the definition. For these $2$-forms, we have the following estimates.
\begin{prop}
Let $g_\varepsilon$ be the K\"ahler metric of $\omega_\varepsilon$ and $\nabla$ be the Levi-Civita connection of $g_\varepsilon$. 
\begin{align*}
\left|\nabla^k\left(\Xi'_{p,i,\varepsilon}|_{\overline{U(p,\varepsilon^{1/2})}}\right)\right|_{g_\varepsilon}&\lesssim \varepsilon^4\sigma_\varepsilon^{-4-k},\\
\left|\nabla^k\left(\Xi'_{p,i,\varepsilon}|_{A(p,\varepsilon^{1/2})}\right)\right|_{g_\varepsilon}&\lesssim \varepsilon^{2-k/2},\\
\nabla^k\left(\Xi'_{p,i,\varepsilon}|_{X\setminus\overline{U(p,2\varepsilon^{1/2})}}\right)&=0.
\end{align*}
\label{prop decay bubble 2 forms}
\end{prop}
\begin{proof}
Since $g_\varepsilon|_{\overline{U(p,\varepsilon^{1/2})}}=\varepsilon^2\mathcal{I}_{p,\varepsilon}^*g_p$, we can see 
\begin{align*}
\left|\nabla^k\left(\Xi'_{p,i,\varepsilon}|_{\overline{U(p,\varepsilon^{1/2})}}\right)\right|_{g_\varepsilon}
&=\varepsilon^2\left|\mathcal{I}_{p,\varepsilon}^*\left(\nabla_{g_p}^k c_1(A_{p,i})\right)\right|_{\varepsilon^2\mathcal{I}_{p,\varepsilon}^*g_p}\\
&=\varepsilon^{-k}\mathcal{I}_{p,\varepsilon}^*\left(\left|\nabla_{g_p}^k c_1(A_{p,i})\right|_{g_p}\right).
\end{align*}
By applying \eqref{ineq nabla k 1} to $\tilde{g}=g_p$ and $g=g_0$, there is a constant $C>0$ such that 
\begin{align*}
|\nabla_{g_p}^kT|_{g_p}\le |\nabla_0^kT|_{g_0}
+C\sum_{j=0}^{k-1}\max\{ 1,r\}^{-4-k+j}|\nabla_0^jT|_{g_0}
\end{align*}
for any tensor $T$ on $M_p$. 
By $({\rm iv})$ of Proposition \ref{prop instanton}, we have 
\begin{align*}
|\nabla_0^kc_1(A_{p,i})|_{g_0}
= O(r^{-4-k})
\end{align*}
as $r\to \infty$. Combining these estimates, there is $C$ such that 
\begin{align*}
|\nabla_{g_p}^kc_1(A_{p,i})|_{g_p}\le C\max\{ 1,r\}^{-4-k}.
\end{align*}
Since $(\mathcal{I}_{p,\varepsilon}^*\max\{ 1,r\})|_{U(p,\varepsilon^{1/2})}=\varepsilon^{-1}\sigma_\varepsilon$, we have 
\begin{align*}
\left|\nabla^k\left(\Xi'_{p,i,\varepsilon}|_{\overline{U(p,\varepsilon^{1/2})}}\right)\right|_{g_\varepsilon}&\lesssim \varepsilon^4\sigma_\varepsilon^{-4-k}.
\end{align*}
The estimate on $A(p,\varepsilon^{1/2})$ with respect to $g_0$ can be shown by Proposition \ref{prop gluing region estimate}. Moreover, applying \eqref{ineq nabla k 1} to $\tilde{g}=g_\varepsilon|_{A(p,\varepsilon^{1/2})}$ and $g=g_0$, we obtain the 2nd estimate.
 The third equality follows directly from the definition. 
\end{proof}

\section{Extension of closed anti-self-dual $2$-forms on $\T$}\label{sec ASD torus}
On the $4$-dimensional torus $\T^4_\Lambda$, the space of harmonic $2$-forms is the $6$-dimensional vector space whose basis is given by 
\begin{align*}
\omega_0&=\frac{\sqrt{-1}}{2}(dz\wedge d\bar{z}+dw\wedge d\bar{w}),\\
{\rm Re}(\Omega_0)&=\frac{1}{2}(dz\wedge dw+d\bar{z}\wedge d\bar{w}),\\
{\rm Im}(\Omega_0)&=\frac{-\sqrt{-1}}{2}(dz\wedge dw-d\bar{z}\wedge d\bar{w}),\\
\omega^-_1&:=\frac{\sqrt{-1}}{2}(dz\wedge d\bar{z}-dw\wedge d\bar{w}),\\
\omega^-_2&:=\frac{1}{2}(dz\wedge d\bar{w}+d\bar{z}\wedge dw),\\
\omega^-_3&:=\frac{\sqrt{-1}}{2}(dz\wedge d\bar{w}-d\bar{z}\wedge dw).
\end{align*}
Here, the space of the self-dual harmonic $2$-forms is generated by $\omega_0,{\rm Re}(\Omega_0),{\rm Im}(\Omega_0)$ and the space of the anti-self-dual harmonic $2$-forms $V^-$ is generated by $\omega^-_1,\omega^-_2,\omega^-_3$. 

In the above sections we have taken a finite subgroup $\Gamma\subset SU(2)$ to obtain $\T=\T(\Lambda,\Gamma)$ and $X=X(\Lambda,\Gamma)$. Since $SU(2)$-action preserves the self-dual $2$-forms, $\omega_0,{\rm Re}(\Omega_0),{\rm Im}(\Omega_0)$ descends to $\T$. Moreover, $\Gamma$ acts on $V^-$ by the pullback of differential forms. Then there is a natural correspondence between the anti-self-dual harmonic $2$-forms on $\T$ and the $\Gamma$-invariant anti-self-dual harmonic $2$-forms. Put $(V^-)^\Gamma=\{ a\in V^-|\, \forall\gamma\in\Gamma,\,\gamma^*a=a\}$. We can show the next proposition easily.
\begin{prop}
$(V^-)^\Gamma\neq \{ 0\}$ iff $\Gamma$ is a subgroup of $\gamma\cdot S^1 \cdot\gamma^{-1}$ for some $\gamma\in \Gamma$. Here, $S^1$ is a subgroup defined by 
\[ S^1=\left\{ \left.
\left (
\begin{array}{ccc}
e^{\sqrt{-1}\theta} & 0 \\
0 & e^{-\sqrt{-1}\theta}
\end{array}
\right )\in SU(2)\right|\, \theta\in\R\right\}.
\]
\end{prop}
Now, we assume that $(V^-)^\Gamma\neq \{ 0\}$. Then by the coordinate transformation by $SU(2)$-action, we can choose the holomorphic coordinate $(z,w)$ on $\C^2$ such that $\Gamma\subset S^1$ and $\omega^-_1\in (V^-)^\Gamma$. Moreover, $\Gamma$ is one of the cyclic groups 
\[ \Z_m:=\left\{ \left.
\left (
\begin{array}{ccc}
e^{2\pi\sqrt{-1}l/m} & 0 \\
0 & e^{-2\pi\sqrt{-1}l/m}
\end{array}
\right )\in SU(2)\right|\, l\in \Z\right\}.
\]
Then we have the next proposition. 
\begin{prop}
Let $\Gamma=\Z_m$ for some $m\ge 2$. Then we have $(V^-)^\Gamma=\R\omega^-_1$ iff $\Gamma=\Z_m$ for $m\ge 3$, and $(V^-)^{\Z_2}=V^-$. 
\end{prop}
From now on, we show that each form in $(V^-)^\Gamma$ can be glued with some harmonic $2$ forms on $M_p$ for every $p\in S$, and we can obtain a $2$-form on $X$.

\subsection{Moment maps on multi-Eguchi-Hanson spaces}
Let $(M,\eta,\Theta,\varpi)$ be an ALE Calabi-Yau manifold asymptotic to $\C^2/\Z_m$ for some $m\ge 2$. It is called the Eguchi-Hanson space if $m=2$, and the multi-Eguchi-Hanson space if $m\ge 3$. They are also called the ALE space of type $A_{m-1}$. In this subsection, we show that all of such spaces have trihamiltonian $G$-action for $G=S^1$ if $m\ge 3$, and $G=SU(2)$ if  $m=2$. If $M$ has a trihamiltonian $G$-action, then there is a hyper-K\"ahler moment map 
\begin{align*}
\mu\colon M\to \R^3\otimes \mathfrak{g}^*,
\end{align*}
where $\mathfrak{g}$ is the Lie algebra of $G$. 

Let $\T^m=(S^1)^m$ be the $m$-torus. 
By Kronheimer's construction in \cite{kronheimer1989construction}, $M$ can be obtained by the hyper-K\"ahler quotient of the $\T^m$-action on $\C^{2m}$ defined by 
\begin{align*}
&\quad\quad
(\lambda_1,\ldots,\lambda_m)\cdot(z_1,w_1,\ldots,z_m,w_m)\\
&\mapsto (\lambda_1\lambda_m^{-1}z_1,\lambda_1^{-1}\lambda_mw_1,\ldots,\lambda_m\lambda_{m-1}^{-1}z_m,\lambda_m^{-1}\lambda_{m-1}w_m).
\end{align*}
We fix an element $\zeta=(\zeta_1,\ldots,\zeta_m)\in\R^m$ such that $\sum_{l=1}^m\zeta_l=0$ and $\zeta_i\neq \zeta_j$ for $i\neq j$, then a level set of hyper-K\"ahler moment map is given by 
\begin{align*}
N_\zeta
&:=\left\{\left. (z_l,w_l)_{l=1}^m\in\C^{2m}\right|\, 
|z_l|^2-|w_l|^2-\zeta_l,\, z_lw_l\mbox{ are independent of }l
\right\},\\
K_\zeta
&:=\{ (z_l,w_l)_{l=1}^m\in N_\zeta|\, z_1\cdots z_m=0\mbox{ or }w_1\ldots w_m=0\},
\end{align*}
and we obtain $M=N_\zeta/\T^m$. Here, $K_\zeta/\T^m$ will be $\varpi^{-1}(0)$. The map $\varpi\colon M\to \C^2/\Z_m$ is given by 
\begin{align*}
(z_l,w_l)_{l=1}^m\mbox{ mod }\T^m\,\,\mapsto\,\,
\left( \sqrt[m]{z_1\cdots z_m},\sqrt[m]{w_1\cdots w_m}\right)\mbox{ mod }\Z_m.
\end{align*}
Here, the choices of $m$-th root of $z_1\cdots z_m$ and $w_1\cdots w_m$ have ambiguity, respectively, we take them such that 
\begin{align*}
\sqrt[m]{z_1\cdots z_m}\cdot\sqrt[m]{w_1\cdots w_m}=z_1w_1.
\end{align*}

Next, we put $G_{S^1}:=S^1$ and define $G_{S^1}$-action on $\C^{2m}$ by 
\begin{align*}
\lambda\cdot(z_l,w_l)_{l=1}^m&:=(\lambda z_l,\lambda^{-1}w_l)_{l=1}^m,
\end{align*}
for $\lambda\in G_{S^1}$. Then $G_{S^1}$, $\T^m$-actions commute with each other, hence $G_{S^1}\times \T^m$ acts on $\C^{2m}$. This action preserves $N_\zeta$ and $G_{S^1}\times \T^m$-action is trihamiltonian with respect to the standard hyper-K\"ahler structure, hence the trihamiltonian $G_{S^1}$-action on $(M,\eta)$ can be defined. 
Moreover, $\varpi$ is $G_{S^1}$-equivariant. Here, the $G_{S^1}$-action on $\C^2$ is defined by 
$\lambda\cdot(z,w)=(\lambda z,\lambda^{-1}w)$ for $\lambda\in G_{S^1}$, and it induces the $G_{S^1}$-action on $\C^2/\Z_m$. 

Denote by $\sqrt{-1}\mu\colon M\to\sqrt{-1}\R$ and $\sqrt{-1}\mu_0\colon \C^2/\Z_m\to \sqrt{-1}\R$ the moment map of $G_{S^1}$-action on $(M,\eta)$ and $(\C^2/\Z_m,\omega_0)$, respectively. By the definition of the moment maps, we have 
\begin{align*}
d\mu=\iota_{{\bm u}}\eta,\quad
d\mu_0=\iota_{{\bm u}_0}\omega_0,
\end{align*}
where $\bm{u},\bm{u}_0$ are the vector fields on $M,\C^2/\Z_m$ generated by $-\sqrt{-1}\in {\rm Lie}(G_{S^1})$, respectively. Note that the moment map is uniquely determined up to additive constants; however, we take $\mu_0$ such that $\mu_0(0)=0$. We also have $\varpi_*(\bm{u})=\bm{u}_0$ and 
\begin{align*}
\bm{u}_0&=2{\rm Im}\left( z\frac{\del}{\del z}-w\frac{\del}{\del w}\right),\\
\mu_0(z,w)&=\frac{1}{2}(|z|^2-|w|^2).
\end{align*}
Here, $(z,w)$ is the standard holomorphic coordinate on $\C^2$ and we consider $\mu_0,\bm{u}_0$ lifting to $\C^2$. Moreover, on $\C^2\setminus\{ 0\}/\Z_m$, we estimate $(\varpi^{-1})^*d\mu-d\mu_0$ and its higher derivatives as follows. Note that 
\begin{align*}
(\varpi^{-1})^*d\mu-d\mu_0
&=\iota_{\bm{u_0}}\left\{ (\varpi^{-1})^*\eta-\omega_0\right\}.
\end{align*}
By Theorem \ref{thm ALE potential}, we can see 
\begin{align*}
\left|\nabla_0^k\left\{ (\varpi^{-1})^*\eta-\omega_0\right\}\right|_{g_0}=O(r^{-4-k}).
\end{align*}
Since $|\bm{u}_0|_{g_0}=O(r)$, $|\nabla_0\bm{u}_0|_{g_0}=O(1)$ and $\nabla_0^k\bm{u}_0\equiv 0$ for $k\ge 2$, we have 
\begin{align*}
\left|\nabla_0^k\left\{ (\varpi^{-1})^*d\mu-d\mu_0\right\}\right|_{g_0}
&=O(r^{-3-k})
\end{align*}
for all $k\ge 0$. In particular, we have $\left| d \left\{ (\varpi^{-1})^*\mu-\mu_0\right\}\right|_{g_0}
=O(r^{-3})$, hence there is a constant $c$ such that $\left| (\varpi^{-1})^*\mu-\mu_0-c\right|=O(r^{-2})$. Now, the moment map $\mu$ has the ambiguity of additive constants, hence we may suppose $c=0$. 
Thus, we obtain the next proposition. 
\begin{prop}
Let $(M,\eta,\Theta,\varpi)$ be an ALE Calabi-Yau manifold asymptotic to $\C^2/\Z_m$ for some $m\ge 2$ and $\mu_0\colon \C^2/\Z_m\to \R$ be defined by $\mu_0(z,w)=\frac{1}{2}(|z^2|-|w|^2)$. 
There is a smooth function $\mu\colon M\to \R$ such that $|\nabla_0^k\{ (\varpi^{-1})^*)\mu-\mu_0\}|=O(r^{-2-k})$ for all $k\ge 0$. 
\label{prop moment S^1}
\end{prop}

Next, we consider the case of $m=2$. In this case, we have the larger symmetry on $M$.

Next we put $G_{SU(2)}=SU(2)$. In the above argument, we define an $G_{SU(2)}$-action on $\C^{2m}$ for $m=2$ by 
\begin{align*}
(a,b)\cdot(z_1,w_1,z_2,w_2)
:=(az_1-\bar{b}w_2,bz_2+\bar{a}w_1,az_2-\bar{b}w_1,bz_1+\bar{a}w_2)
\end{align*}
for $(a,b)\in G_{SU(2)}\cong S^3=\{ (u,v)\in\C^2|\, |u|^2+|v|^2=1\}$ and $(z_1,w_1,z_2,w_2)\in\C^4$. 
Then the $G_{SU(2)}$-action commutes with the $\T^2$-action and preserves the standard hyper-K\"ahler structure on $\C^2$, hence also preserves the hyper-K\"ahler moment map and its level set $N_\zeta$. Therefore, we can induce the $G_{SU(2)}$-action on $M$ and it is hamiltonian with respect to $\eta$. Moreover, the standard $G_{SU(2)}$-action on $\C^2$ descends to the action on $\C^2/\Z_2$, and we can see that 
\[ \varpi((a,b)(z_1,w_1,z_2,w_2))=
\left (
\begin{array}{ccc}
a & -\bar{b} \\
b & \bar{a}
\end{array}
\right )\varpi(z_1,w_1,z_2,w_2),
\]
hence $\varpi$ is $G_{SU(2)}$-equivariant. 

Here, we have identified $G_{SU(2)}$ with $S^3$, its Lie algebra can be identified with $T_{(1,0)}S^3=\R(\sqrt{-1},0)\oplus\R(0,1)\oplus\R(0,\sqrt{-1})$. Denote by $\bm{v}_1,\bm{v}_2,\bm{v}_3$ the vector fields on $M$ generated by $(-\sqrt{-1},0),(0,-1),(0,-\sqrt{-1})$, respectively and put $\bm{v}_{0,\alpha}=\varpi_*\bm{v}_\alpha$. Then we may write 
\begin{align*}
\bm{v}_{0,1}
&=2{\rm Im}\left( z\frac{\del}{\del z}-w\frac{\del}{\del w}\right),\\
\bm{v}_{0,2}
&=-2{\rm Re}\left( z\frac{\del}{\del w}-w\frac{\del}{\del z}\right),\\
\bm{v}_{0,3}
&=2{\rm Im}\left( w\frac{\del}{\del z}+z\frac{\del}{\del w}\right).
\end{align*}
Let $\nu_\alpha,\nu_{0,\alpha}$ be the hamiltonian functions for $\bm{v}_\alpha,\bm{v}_{0,\alpha}$, respectively, defined by 
\begin{align*}
d\nu_\alpha=\iota_{\bm{v}_\alpha}\eta,\quad
d\nu_{0,\alpha}=\iota_{\bm{v}_{0,\alpha}}\omega_0.
\end{align*}
Here,  we take $\bm{v}_{0,\alpha}$ such that $\bm{v}_{0,\alpha}(0)=0$, then we may write 
\begin{align*}
\nu_{0,1}(z,w)
&=\frac{1}{2}(|z|^2-|w|^2),\\
\nu_{0,2}(z,w)
&=-{\rm Im}\left( \bar{z}w\right),\\
\nu_{0,3}(z,w)
&={\rm Re}\left( \bar{z}w\right).
\end{align*}
By the same argument for $\mu,\mu_0$, we have the next proposition. 
\begin{prop}
Let $(M,\eta,\Theta,\varpi)$ be an ALE Calabi-Yau manifold asymptotic to $\C^2/\Z_2$ and $\nu_{0,\alpha}\colon \C^2/\Z_2\to \R$ be defined by $\nu_{0,1}=\frac{1}{2}(|z^2|-|w|^2)$, 
$-\sqrt{-1}\nu_{0,2}+\nu_{0,3}=\bar{z}w$. 
There is a smooth function $\nu_\alpha\colon M\to \R$ such that $|\nabla_0^k\{ (\varpi^{-1})^*\nu_\alpha-\nu_{0,\alpha}\}|=O(r^{-2-k})$ for all $\alpha=1,2,3$ and $k\ge 0$. 
\label{prop moment SU(2)}
\end{prop}

\subsection{Anti-self-duality of $\sqrt{-1}\ddb\mu$ and $\sqrt{-1}\ddb\nu_\alpha$}
In this subsection, we show that $\sqrt{-1}\ddb\mu$ and $\sqrt{-1}\ddb\nu_\alpha$ obtained in Propositions \ref{prop moment S^1} and \ref{prop moment SU(2)} are anti-self-dual. 
Let $(M,g,I,J,K)$ be a hyper-K\"ahler manifold, that is to say, $g$ is a Riemannian metric, $I,J,K$ are  complex structures with $IJK=-1$, $g$ is hermitian with respect to $I,J,K$ and $\omega_I=g(I\cdot,\cdot)$, $\omega_J=g(J\cdot,\cdot)$, $\omega_K=g(K\cdot,\cdot)$ are closed. 
Suppose that $\bm{u}\in\mathcal{X}(M)$ is a vector field preserving $g,I,J,K$. Then the flow $\{ \exp(t\bm{u})\}_{t\in\R}$ gives an $\R$-action on $(M,g,I,J,K)$, and we assume this action is trihamiltonian. Therefore, there is a smooth map $F=(F_1,F_2,F_3)\colon M\to \R^3$ such that $dF_1=\iota_{\bm{u}}\omega_I$, $dF_2=\iota_{\bm{u}}\omega_J$, $dF_3=\iota_{\bm{u}}\omega_K$.

It follows from the next Lemma. 
\begin{lem}
Let $\dim_\R M=4$ and $M$ be a complex manifold with respect to $I$. Then $\sqrt{-1}\ddb F_1$ is an anti-self-dual $2$-form.
\label{lem ASD}
\end{lem}
\begin{proof}
We may write $\sqrt{-1}\ddb f=-d(df\circ I)/2$ for any function $f$. Since $dF_1=\iota_{\bm{u}}\omega_I$, we have $dF_1\circ I=g(\bm{u},\cdot)$. Therefore, we have $dF_1\circ I=dF_2\circ J=dF_3\circ K$. 
Since $d(dF_1\circ I)$ is of type $(1,1)$ with respect to $I$, it is of type $(1,1)$ with respect to $J,K$. In the hyper-K\"ahler manifolds of real dimension $4$, $(1,1)$-forms with respect to $I,J,K$ are always anti-self-dual. 
\end{proof}

\subsection{Closed $(1,1)$-forms on $X$ coming from $\omega_\alpha^-$}
Here, we glue $\omega_\alpha^-$ with $\sqrt{-1}\ddb\nu_\alpha$. Notice that $\omega_\alpha^-=\sqrt{-1}\ddb\nu_{0,\alpha}$. 
Let $(\T,S)\approx (M_p,\eta_p,\Theta_p,\varpi_p)_{p\in S}$. 
If $\Gamma_p=\Z_m$ for $m\ge 3$, 
denote by $\nu_{p,1}\in C^\infty(M_p)$ the functions obtained by Proposition \ref{prop moment S^1}. If $\Gamma_p=\Z_2$, 
denote by $\nu_{p,\alpha}\in C^\infty(M_p)$ the functions obtained by Proposition \ref{prop moment SU(2)}. 

Now, we define $\omega^-_{\alpha,\varepsilon}\in\Omega^{1,1}(X)$ by 
\begin{align*}
\omega^-_{\alpha,\varepsilon}|_{\overline{U(p,\varepsilon^{1/2})}}&:=\sqrt{-1}\varepsilon^2\mathcal{I}_{p,\varepsilon}^*\ddb\nu_{p,\alpha},\\
\omega^-_{\alpha,\varepsilon}|_{A(p,\varepsilon^{1/2})}&:=\pi^*\omega_\alpha^-
+\sqrt{-1}\ddb \left( \chi_{\varepsilon^{1/2}}\cdot \varepsilon^2\mathcal{I}_{p,\varepsilon}^*(\nu_{p,\alpha}-\varpi_p^*\nu_{0,\alpha})\right),\\
\omega^-_{\alpha,\varepsilon}|_{X\setminus\overline{\bigcup_pU(p,2\varepsilon^{1/2})}}&:=\pi^*\omega_\alpha^-,
\end{align*}
which are closed by the definition. 
The next proposition can be shown by the same argument to show Proposition \ref{prop decay bubble 2 forms}.
\begin{prop}\label{prop decay outside 2 forms}
Let $g_\varepsilon$ be the K\"ahler metric of $\omega_\varepsilon$ and $\nabla$ be the Levi-Civita connection of $g_\varepsilon$. 
\begin{align*}
\left|\nabla^k\left(\omega^-_{\alpha,\varepsilon}|_{\overline{U(p,\varepsilon^{1/2})}}\right)\right|_{g_\varepsilon}&\lesssim \varepsilon^4\sigma_\varepsilon^{-4-k},\\
\left|\omega^-_{\alpha,\varepsilon}|_{A(p,\varepsilon^{1/2})}\right|_{g_\varepsilon}&\lesssim 1,\\
\left|\nabla^k\left(\omega^-_{\alpha,\varepsilon}|_{A(p,\varepsilon^{1/2})}\right)\right|_{g_\varepsilon}&\lesssim \varepsilon^{2-k/2}\quad (k\ge 1),\\
\left|\omega^-_{\alpha,\varepsilon}|_{X\setminus\overline{U(p,2\varepsilon^{1/2})}}\right|_{g_\varepsilon}&={\rm constant},\\
\nabla^k\left(\omega^-_{\alpha,\varepsilon}|_{X\setminus\overline{U(p,2\varepsilon^{1/2})}}\right)&=0\quad (k\ge 1),\\
\left|\omega_{\varepsilon,\alpha}^-\wedge\omega_\varepsilon|_{A(p,\varepsilon^{1/2})}\right|_{g_\varepsilon}
&\lesssim \varepsilon^2.
\end{align*}
\end{prop}
\begin{proof}
The proof is almost the same as the proof of Proposition \ref{prop decay bubble 2 forms}. 
The difference is that we have to estimate $\nabla^k(\pi^*\omega_\alpha^-)$ on $A(p,\varepsilon^{1/2})$ and $X\setminus\overline{U(p,2\varepsilon^{1/2})}$. 
Since $g_\varepsilon=g_0$ on $X\setminus\overline{U(p,2\varepsilon^{1/2})}$, hence $\pi^*\omega_\alpha^-$ is parallel with respect to $g_\varepsilon$ on this region. It gives the $4$th and $5$th equality. On $A(p,\varepsilon^{1/2})$, we have $g_0/2\le g_\varepsilon \le 2g_\varepsilon$ and $|\nabla_0^k\omega_\varepsilon|_{g_0}\lesssim \varepsilon^{2-k/2}$ for sufficiently small $\varepsilon$ and $k\ge 1$ by \eqref{eq decay gluing1}. Applying \eqref{ineq nabla k 1} to $\tilde{g}=g_\varepsilon$ and $g=g_0$, we have $|\nabla^k(\pi^*\omega_\alpha^-)|_{g_\varepsilon}\lesssim \varepsilon^{2-k/2}$ on $A(p,\varepsilon^{1/2})$, since $\nabla_0(\pi^*\omega_\alpha^-)=0$ and $|\pi^*\omega_\alpha^-|_{g_0}$ is constant. 
The $6$th estimate follows from $\pi^*\omega_\alpha^-\wedge\pi^*\omega_0=0$ and $|\omega^-_{\alpha,\varepsilon} - \pi^*\omega_\alpha^-|_{g_\varepsilon}\lesssim \varepsilon^2$, $|\omega_\varepsilon-\pi^*\omega_0|_{g_\varepsilon}\lesssim \varepsilon^2$ on $A(p,\varepsilon^{1/2})$. 
\end{proof}

\section{Cohomology groups}
Here, we give the basis of $H_2(X,\R)$ and compute the cohomology classes represented by the closed $2$-forms constructed in the previous sections.

By the arguments in Subsections \ref{subsec kummer} and \ref{subsec appro}, 
We have the decomposition $X=(\T\setminus S)\cup \bigsqcup_{p\in S}U(p,t)$ for a sufficiently small $t>0$ as smooth manifolds, and $(\T\setminus S)\cap \bigsqcup_{p\in S}U(p,t)$ is homotopy equivalent to $\bigsqcup_{p\in S}S^3/\Gamma_p$. Then, by the Mayer-Vietoris exact sequence for the de Rham cohomology groups and $H^1(S^3/\Gamma_p)=H^3(S^3/\Gamma_p)=\{ 0\}$, we have the following isomorphism 
\begin{align*}
i_{\T\setminus S}^*\oplus \bigoplus_{p\in S}i_{U(p,t)}^*\colon H^2(X)\to H^2(\T\setminus S)\oplus\left( \bigoplus_{p\in S} H^2(U(p,t))\right).
\end{align*}
Here, for a subset $S\subset X$, $i_S\colon S\to X$ is the inclusion map. 
Denote by $\pi_\Gamma\colon \T^4_\Lambda\to \T$ the quotient map, then we have 
$\T^4_\Lambda=(\T^4_\Lambda\setminus\pi_\Gamma^{-1}(S))\cup\bigsqcup^{\# \pi_\Gamma^{-1}(S))}B^4$, where $B^4\subset \C^2$ is the open unit ball. Here, we have $(\T^4_\Lambda\setminus\pi_\Gamma^{-1}(S))\cap\bigsqcup^{\# \pi_\Gamma^{-1}(S))}B^4=\bigsqcup^{\# \pi_\Gamma^{-1}(S))}S^3$. Then, by the Mayer-Vietoris exact sequence, we have the isomorphism 
\begin{align*}
i_{\T^4_\Lambda\setminus\pi_\Gamma^{-1}(S)}^*\colon H^2(\T^4_\Lambda)\to H^2(\T^4_\Lambda\setminus\pi_\Gamma^{-1}(S)).
\end{align*}
Now, we have the natural $\Gamma$-actions on $H^2(\T^4_\Lambda)$ and $ H^2(\T^4_\Lambda\setminus\pi_\Gamma^{-1}(S))$ such that the above isomorphism is $\Gamma$-equivariant, hence we obtain the isomorphisms between $\Gamma$-invariant subspaces 
\begin{align*}
H^2(\T^4_\Lambda)^\Gamma\stackrel{\cong}{\rightarrow} H^2(\T^4_\Lambda\setminus\pi_\Gamma^{-1}(S))^\Gamma\cong H^2(\T\setminus S).
\end{align*}
Next, we consider $H^2(U(p,t))$. By the identification $U(p,t)\cong \varpi_p^{-1}(B_{\Gamma_p}(t))$, we have the inclusion $i_t\colon U(p,t)\hookrightarrow M_p$. Since $U(p,t)\setminus \overline{U(p,t/2)}\cong (S^3/\Gamma_p)\times (t/2,t)$ and $M_p\setminus \overline{U(p,t/2)}\cong (S^3/\Gamma_p)\times (t/2,+\infty)$ and we can take a smooth increasing function $\tau\colon (t/2,t)\to (t/2,+\infty)$ such that $\tau(t/2)=t/2$, $\lim_{s\to t}\tau(s)=+\infty$, then $i_t$ is a homotopy equivalence. In particular, $i_t^*\colon H^2(M_p)\to H^2(U(p,t))$ is an isomorphism.

Combining these identifications, we have the following isomorphism
\begin{align}
\mathcal{J}\colon H^2(X)\stackrel{\cong}{\rightarrow} H^2(\T^4_\Lambda)^\Gamma\oplus\left(\bigoplus_{p\in S} H^2(M_p)\right).\label{eq def J}
\end{align}
For a $2$-form $\alpha$, denote by $[\alpha]$ its cohomology class. Then 
\begin{align*}
\{ [\omega_0],\,[{\rm Re}(\Omega_0)],\,[{\rm Im}(\Omega_0)],\, [\omega^-_\alpha]|\, \omega^-_\alpha\in (V^-)^\Gamma\}
\end{align*}
is a basis of $H^2(\T^4_\Lambda)^\Gamma$, and $\{ [c_1(A_{p,i})]|\, i=1,\ldots,N_{\Gamma_p}\}$ is a basis of $H^2(M_p)$ by $({\rm ii})$ of Theorem \ref{thm ASD on ALE}. Moreover, for 
\begin{align*}
a\in H^2(\T^4_\Lambda)^\Gamma\oplus\left(\bigoplus_{p\in S} H^2(M_p)\right),
\end{align*}
we put 
\begin{align*}
a_X:=\mathcal{J}^{-1}(a).
\end{align*} 
For example, if $\alpha\in \Omega^2(M_p)$ is closed, then $[\alpha]$ is a cohomology class on $M_p$ and $[\alpha]_X=\mathcal{J}^{-1}([\alpha])$ is a cohomology class on $X$. Since $\mathcal{J}$ is an isomorphism, 
\begin{align*}
[\omega_0]_X\,[{\rm Re}(\Omega_0)]_X,\,[{\rm Im}(\Omega_0)]_X,\, [\omega^-_\alpha]_X,\, [c_1(A_{p,i})]_X
\end{align*}
form a basis of $H^2(X)$.

\begin{prop}
We have 
\begin{align*}
[\omega_\varepsilon]&=[\omega_0]_X+\varepsilon^2\sum_{p\in S}[\eta_p]_X,\\
[{\rm Re}(\Omega)]&=[{\rm Re}(\Omega_0)]_X,\quad
[{\rm Im}(\Omega)]=[{\rm Im}(\Omega_0)]_X,\\
[\Xi'_{p,i,\varepsilon}]&=\varepsilon^2[c_1(A_{p,i})]_X,\quad
[\omega^-_{\alpha,\varepsilon}]=[\omega^-_\alpha]_X.
\end{align*}
In particular, the cohomology classes $\varepsilon^{-2}[\Xi'_{p,i,\varepsilon}],[\omega^-_{\alpha,\varepsilon}]$ are independent of $\varepsilon$. 
\label{prop decomp H2X}
\end{prop}
\begin{proof}
Let $j_t\colon \T\setminus \bigcup_{p\in S}B_\T(p,t))\to \T\setminus S$ be the inclusion map. By a similar argument to show that $i_t^*$ is an isomorphism, we can see that $j_t^*\colon H^2(\T\setminus S)\to H^2(\T\setminus \bigcup_{p\in S}B_\T(p,t))$ is an isomorphism for sufficiently small $t>0$. 
By the construction of $\omega_\varepsilon$, we can see $j_t^*\omega_\varepsilon=j_t^*\pi^*\omega_0$ for $t\ge 2\varepsilon^{1/2}$. Therefore, we have $[\omega_\varepsilon]|_{\T\setminus S}=[\omega_0]|_{\T\setminus S}$. If we take $t\le \varepsilon^{1/2}$, then $i_t^*\omega_\varepsilon=\varepsilon^2i_t^*\mathcal{I}_{p,\varepsilon}^*\eta_p=\varepsilon^2i_t^*\mathbf{H}_{\varepsilon^{-1}}^*\eta_p$. Therefore, $i_t^*[\omega_\varepsilon]=\varepsilon^2i_t^*[\eta_p]$, which gives $\mathcal{J}(\omega_\varepsilon)=[\omega_0]+\varepsilon^2\sum_{p\in S}[\eta_p]$. We can show $\mathcal{J}([\Xi'_{p,i,\varepsilon}])=\varepsilon^2c_1(A_{p,i})$ in the same way. Moreover, $\mathcal{J}([\omega^-_{\alpha,\varepsilon}])=[\omega^-_\alpha]$ follows since $\omega^-_{\alpha,\varepsilon}|_{U(p,\varepsilon^{1/2})}$ is a exact form. To show $\mathcal{J}([{\rm Re}(\Omega)])=[{\rm Re}(\Omega_0)]$ and $
\mathcal{J}([{\rm Im}(\Omega)])=[{\rm Im}(\Omega_0)]$, it suffices to see that $[\Omega|_{U(p,t)}]=0\in H^2(U(p,t))$ for sufficiently small $t>0$. By the argument in \cite{KN1990instanton}, we may write $\varpi_p^{-1}(0)=\bigcup_i E_i$, where $E_i\subset M_p$ is a $(-2)$-curve and $[E_1],\ldots,[E_{N_{\Gamma_p}}]\in H_2(M_p,\Z)$ form a dual basis of $c_1(A_{p,1}),\ldots,c_1(A_{p,N_{\Gamma_p}})\in H^2(M_p,\Z)$. Since $\Omega$ is a $(2,0)$-form, we have $\int_{E_i}\Omega=0$ for all $i$, which implies $[\Omega|_{U(p,t)}]=0$. 
\end{proof}

Next, we compute the cup product on $H^2(X)$. For closed $2$-forms $\alpha,\beta\in \Omega^2(X)$, we put $[\alpha]\cdot[\beta]:=\int_X\alpha\wedge\beta$. We put 
\begin{align*}
H_+^2(\T^4_\Lambda)&:={\rm span}\{ [\omega_0],[{\rm Re}(\Omega_0)], [{\rm Im}(\Omega_0)]\},\\
H_-^2(\T^4_\Lambda)^\Gamma&:=\{ [\alpha]|\, \alpha\in  (V^-)^\Gamma\},
\end{align*}
then we have the decomposition $H^2(\T^4_\Lambda)^\Gamma=H_+^2(\T^4_\Lambda)\oplus
H_-^2(\T^4_\Lambda)^\Gamma$.
\begin{prop}\label{prop cup prod}
We have 
\begin{align*}
[\omega_0]_X^2&=[{\rm Re}(\Omega_0)]_X^2=[{\rm Im}(\Omega_0)]_X^2={\rm vol}(\T),\\
[\omega_0]_X\cdot[\Omega_0]_X&=[{\rm Re}(\Omega_0)]_X\cdot[{\rm Im}(\Omega_0)]_X=0,\\
[c_1(A_{p,i})]_X\cdot[\Omega_0]_X
&=[\omega^-_\alpha]_X\cdot[\Omega_0]_X=0,\\
[\omega^-_\alpha]_X\cdot[\omega^-_\beta]_X
&= -\delta_{\alpha\beta}{\rm vol}(\T),\\
[c_1(A_{p,i})]_X\cdot[c_1(A_{p,j})]_X &= \int_{M_p}c_1(A_{p,i})\wedge c_1(A_{p,j}),\\
[c_1(A_{p,i})]_X\cdot[c_1(A_{q,j})]_X&=0\quad(p\neq q),\\
[c_1(A_{p,i})]_X\cdot[\omega^-_\alpha]_X&=0,\\
[\omega_0]_X\cdot[c_1(A_{p,i})]_X &= [\omega_0]_X\cdot[\omega^-_\alpha]_X =0.
\end{align*}
\end{prop}
\begin{proof}
We can see $[\omega_\varepsilon]\cdot[\Omega]=[{\rm Re}(\Omega)]\cdot[{\rm Im}(\Omega)]=0$ and $[\Xi'_{p,i,\varepsilon}]\cdot[\Omega]=[\omega^-_{\alpha,\varepsilon}]\cdot[\Omega]=0$. Therefore, we have 
\begin{align*}
\left( [\omega_0]_X+\varepsilon^2Q\right)\cdot[\Omega_0]_X&=[{\rm Re}(\Omega_0)]_X\cdot[{\rm Im}(\Omega_0)]_X=0,\\
[c_1(A_{p,i})]_X\cdot[\Omega_0]_X
&=[\omega^-_\alpha]_X\cdot[\Omega_0]_X=0,
\end{align*}
where $Q=\sum_p[\eta_p]_X$. Since the above equality holds for all $\varepsilon$, we also have 
$[\omega_0]_X\cdot[\Omega_0]_X=Q\cdot[\Omega_0]_X=0$. Moreover, we can check $[\omega_0]_X^2=\lim_{\varepsilon\to 0}[\omega_\varepsilon]^2=\int_{\T\setminus S}\omega_0^2={\rm vol}(\T)=[{\rm Re}(\Omega_0)]_X^2=[{\rm Im}(\Omega_0)]_X^2$.
It is easy to see $[\Xi'_{p,i,\varepsilon}]\cdot[\Xi'_{q,j,\varepsilon}]=0$ if $p\neq q$ and $\varepsilon\ll 1$, hence $[c_1(A_{p,i})]_X\cdot[c_1(A_{q,j})]_X=0$.

Next, we have 
\begin{align*}
[c_1(A_{p,i})]_X\cdot[c_1(A_{p,j})]_X&=\varepsilon^{-4}\int_X\Xi'_{p,i,\varepsilon}\wedge \Xi'_{p,j,\varepsilon}\\
&=\varepsilon^{-4}\left( \int_{U(p,\varepsilon^{1/2})}\Xi'_{p,i,\varepsilon}\wedge \Xi'_{p,j,\varepsilon}+\int_{A(p,\varepsilon^{1/2})}\Xi'_{p,i,\varepsilon}\wedge \Xi'_{p,j,\varepsilon}\right).
\end{align*}
Now, we have 
\begin{align*}
\int_{U(p,\varepsilon^{1/2})}\Xi'_{p,i,\varepsilon}\wedge \Xi'_{p,j,\varepsilon}
&=\varepsilon^4\int_{\varpi_p^{-1}(B_{\Gamma_p}(\varepsilon^{-1/2}))}c_1(A_{p,i})\wedge c_1(A_{p,j})\\
&=\varepsilon^4\bigg( \int_{M_p}c_1(A_{p,i})\wedge c_1(A_{p,j}) \\
&\quad\quad
- \int_{M_p\setminus \varpi_p^{-1}(B_{\Gamma_p}(\varepsilon^{-1/2}))}c_1(A_{p,i})\wedge c_1(A_{p,j})\bigg),
\end{align*}
and by Proposition \ref{prop decay bubble 2 forms} and $({\rm iv})$ of Proposition \ref{prop instanton}, we have 
\begin{align*}
\left|\int_{M_p\setminus \varpi_p^{-1}(B_{\Gamma_p}(\varepsilon^{-1/2}))}c_1(A_{p,i})\wedge c_1(A_{p,j})\right|
&\lesssim \int_{\varepsilon^{-1/2}}^\infty r^{-8}\cdot r^3dr\lesssim \varepsilon^2\\
\left|\int_{A(p,\varepsilon^{1/2})}\Xi'_{p,i,\varepsilon}\wedge \Xi'_{p,j,\varepsilon}\right|
&\lesssim \varepsilon^6.
\end{align*}
Therefore, 
\begin{align*}
\left|[c_1(A_{p,i})]_X\cdot[c_1(A_{p,j})]_X - \int_{M_p}c_1(A_{p,i})\wedge c_1(A_{p,j})\right|&\lesssim \varepsilon^2.
\end{align*}
Since the left-hand side is independent of $\varepsilon$, it should be $0$. It is easy to check that 
\begin{align*}
\lim_{\varepsilon\to 0}[\omega^-_{\alpha,\varepsilon}]\cdot[\omega^-_{\beta,\varepsilon}]
&= \int_{\T\setminus S}\omega^-_\alpha\wedge\omega^-_\beta,
\end{align*}
by Proposition \ref{prop decay outside 2 forms},
hence $[\omega^-_\alpha]_X\cdot[\omega^-_\beta]_X
= -\delta_{\alpha\beta}{\rm vol}(\T)$. 
Since 
\begin{align*}
[c_1(A_{p,i})]\cdot[\omega^-_\alpha]
&=\varepsilon^{-2}\int_X\Xi'_{p,i,\varepsilon}\wedge\omega^-_{\alpha,\varepsilon}\\
&=\varepsilon^{-2}\int_{U(p,2\varepsilon^{1/2})}\Xi'_{p,i,\varepsilon}\wedge\omega^-_{\alpha,\varepsilon},
\end{align*}
then, by Propositions \ref{prop decay bubble 2 forms} and  \ref{prop decay outside 2 forms}, we have 
\begin{align*}
\left|\int_{A(p,\varepsilon^{1/2})}\Xi'_{p,i,\varepsilon}\wedge\omega^-_{\alpha,\varepsilon}\right|
&\lesssim \varepsilon^4,\\
\left|\int_{U(p,\varepsilon^{1/2})}\Xi'_{p,i,\varepsilon}\wedge\omega^-_{\alpha,\varepsilon}\right|
&\lesssim \varepsilon^8\int_\varepsilon^{\varepsilon^{1/2}}r^{-5}dr+\varepsilon^4\lesssim \varepsilon^4.
\end{align*}
Therefore, 
\begin{align*}
|[c_1(A_{p,i})]_X\cdot[\omega^-_\alpha]_X|\lesssim\varepsilon^2.
\end{align*}
Since $[c_1(A_{p,i})]\cdot[\omega^-_\alpha]$ does not depend on $\varepsilon$, we have $[c_1(A_{p,i})]\cdot[\omega^-_\alpha]=0$. Next, we consider $([\omega_0]+\varepsilon^2Q)\cdot[\omega^-_\alpha]$. 
Since $\omega_\varepsilon$ is self-dual outside $A(p,\varepsilon)$ and $\omega^-_{\alpha,\varepsilon}$ is anti-self-dual outside $A(p,\varepsilon)$ by Lemma \ref{lem ASD}, we can see 
\begin{align*}
\left|([\omega_0]_X+\varepsilon^2Q)\cdot[\omega^-_\alpha]_X\right|
\le\sum_{p\in S}\left|\int_{A(p,\varepsilon^{1/2})}\omega_\varepsilon\wedge\omega^-_{\alpha,\varepsilon}\right|\lesssim \varepsilon^2,
\end{align*}
hence $[\omega_0]_X\cdot[\omega^-_\alpha]_X=0$. 
We can also see $\omega_\varepsilon\wedge\Xi'_{p,i\varepsilon}=0$ outside $A(p,\varepsilon^{1/2})$, we have 
\begin{align*}
\left|([\omega_0]_X+\varepsilon^2Q)\cdot[c_1(A_{p,i})]_X\right|
\le\varepsilon^{-2}\sum_{p\in S}\left|\int_{A(p,\varepsilon^{1/2})}\omega_\varepsilon\wedge\Xi'_{p,i,\varepsilon}\right|\lesssim \varepsilon^2.
\end{align*}
As $\varepsilon\to 0$, we have $[\omega_0]_X\cdot[c_1(A_{p,i})]_X = 0$.
\end{proof}
By $({\rm i})({\rm ii})$ of Proposition \ref{prop instanton}, the $N_{\Gamma_p}\times N_{\Gamma_p}$-matrix 
\begin{align*}
\left( L_{p,ij}\right)_{i,j}:=\left( \int_{M_p}c_1(A_{p,i})\wedge *_{\eta_p}c_1(A_{p,j})\right)_{i,j}
\end{align*}
is symmetric and positive definite. Since $c_1(A_{p,j})$ is anti-self dual, we may write 
\begin{align*}
[c_1(A_{p,i})]_X\cdot[c_1(A_{p,j})]_X &= -L_{p,ij}.
\end{align*}
Let $(L_p^{ij})_{i,j}$ be the inverse matrix of $(L_{p,ij})_{i,j}$. By Theorem \ref{thm ASD on ALE}, $[c_1(A_{p,1})],\ldots,[c_1(A_{p,\Gamma_p})]$ form a basis of $H^2(M_p)$. 
Since $[\eta_p]\in H^2(M_p)$, we can take $\eta_p^i \in \R$ such that 
\begin{align*}
[\eta_p]=\sum_{i=1}^{N_{\Gamma_p}}\eta_p^i [c_1(A_{p,i})]\in H^2(M_p).
\end{align*}
Then we have 
\begin{align*}
[\eta_p]_X\cdot [c_1(A_{p,i})]_X&=-\sum_{j=1}^{N_{\Gamma_p}}\eta_p^j L_{p,ji},
\end{align*}
and if we put $Q=\sum_{p\in S}[\eta_p]_X$, then 
\begin{align*}
Q^2&=-\sum_{p\in S}\sum_{i,j=1}^{N_{\Gamma_p}}\eta_p^i\eta_p^jL_{p,ij}.
\end{align*}
\begin{rem}
\normalfont
Here, the constants $L_{p,ij}$ are determined topologically, do not depend on the Ricci-flat K\"ahler metric $\eta_p$ by the following argument. We have $\gamma_{p,i}\in\Omega^1(M_p\setminus\varpi_p^{-1}(0))$ such that $c_1(A_{p,i})=d\gamma_{p,i}$ by Proposition \ref{prop instanton}. Let $\chi_R\in C^\infty(M_p)$ be a smooth cut-off function such that ${\rm supp}(\chi_R)\subset \varpi_p^{-1}(B_{\Gamma_p}(R+1))$ and $\chi_R|_{\varpi_p^{-1}(B_{\Gamma_p}(R))}\equiv 1$. Put $F_{p,i,R}:=d(\chi_R\gamma_{p,i})$. Then a compactly supported cohomology class $[F_{p,i,R}]_c\in H^2_c(M_p)$ does not depend on $R$, hence $\int_{M_p}F_{p,i,R}\wedge F_{p,j,R}$ is also independent of $R$. Moreover, by the estimates in Proposition \ref{prop instanton}, we have $\lim_{R\to \infty}\int_{M_p}F_{p,i,R}\wedge F_{p,j,R}=\int_{M_p}c_1(A_{p,i})\wedge c_1(A_{p,j})=-L_{p,ij}$. Therefore, $L_{p,ij}$ are determined by $[F_{p,i}]_c,[F_{p,j}]_c$. 
\end{rem}


Recall that $W_\varepsilon=2[\omega_\varepsilon]^2/[\Omega][\overline{\Omega}]$. By Proposition \ref{prop cup prod}, we may write 
\begin{align*}
W_\varepsilon=\frac{([\omega_0]_X+\varepsilon^2Q)^2}{[\omega_0]_X^2}=1+\varepsilon^4\frac{Q^2}{{\rm vol}(\T)}.
\end{align*}
Here, we have $[\omega_0]_X\cdot Q=0$ since $Q$ is written as a linearly combination of $\{ c_1(A_{p,i})\}_{p,i}$. Therefore, 
\begin{align}
[\tilde{\omega}_\varepsilon]=W_\varepsilon^{-1/2}[\omega_\varepsilon]=\frac{[\omega_0]_X+\varepsilon^2Q}{\sqrt{1+\varepsilon^4Q^2/{\rm vol}(\T)}}.\label{eq RFK class}
\end{align}


\section{Main result}\label{sec main}
Let $\tilde{\omega}_\varepsilon$ be the Ricci-flat K\"ahler forms obtained by Theorem \ref{thm RFK and appro}. Denote by $\mathcal{H}^2_{\tilde{\omega}_\varepsilon}$ the space of harmonic $2$-forms on $X$ with respect to the K\"ahler metric $\tilde{\omega}_\varepsilon$. Moreover, we have the orthogonal decomposition $\mathcal{H}^2_{\tilde{\omega}_\varepsilon}=\mathcal{H}^2_{\tilde{\omega}_\varepsilon,+}\oplus \mathcal{H}^2_{\tilde{\omega}_\varepsilon,-}$, where $\mathcal{H}^2_{\tilde{\omega}_\varepsilon,\pm}=\{ \zeta\in\mathcal{H}^2_{\tilde{\omega}_\varepsilon}|\, *\zeta=\pm \zeta\}$, respectively.
\begin{thm}\label{thm main 7}
Let $d_\Gamma=\dim (V^-)^\Gamma$ and $\omega^-_1,\ldots,\omega^-_{d_\Gamma}$ be the orthonormal basis of $(V^-)^\Gamma$. 
There is a family of closed $2$-forms on $X$ 
\begin{align*}
\left( \tilde{\omega}_\varepsilon,{\rm Re}(\Omega),{\rm Im}(\Omega),(\tilde{\omega}^-_{\alpha,\varepsilon})_{\alpha=1,\ldots,d_\Gamma},(\tilde{\Xi}_{p,i,\varepsilon})_{p\in S,i=1,\ldots,N_{\Gamma_p}}\right)_{0<\varepsilon\le \varepsilon_0}
\end{align*}
satisfying the following properties. 
\begin{itemize}
\setlength{\parskip}{0cm}
\setlength{\itemsep}{0cm}
 \item[$({\rm i})$] We have 
\begin{align*}
[\tilde{\omega}_\varepsilon]&=\frac{[\omega_0]_X+\varepsilon^2Q}{\sqrt{1+\varepsilon^4Q^2/{\rm vol}(\T)}},\\
[{\rm Re}(\Omega)]&=[{\rm Re}(\Omega_0)]_X,\quad
[{\rm Im}(\Omega)])=[{\rm Im}(\Omega_0)]_X,\\
[\tilde{\Xi}_{p,i,\varepsilon}]&=[c_1(A_{p,i})]_X-\frac{\varepsilon^2[c_1(A_{p,i})]_X\cdot Q}{{\rm vol}(\T)}[\omega_0]_X,\\
[\tilde{\omega}^-_{\alpha,\varepsilon}]&=[\omega^-_\alpha]_X.
\end{align*}
In particular, the cohomology classes $[\tilde{\omega}^-_{\alpha,\varepsilon}]$ are independent of $\varepsilon$.
 \item[$({\rm ii})$] $\tilde{\omega}_\varepsilon$ satisfies the Monge-Amp\`ere equation $\tilde{\omega}_\varepsilon^2=\frac{1}{2}\Omega\wedge\bar{\Omega}$. In particular, $\tilde{\omega}_\varepsilon/\sqrt{\vol(\T)},{\rm Re}(\Omega)/\sqrt{\vol(\T)},{\rm Im}(\Omega)/\sqrt{\vol(\T)}$ form an orthonormal basis of $\mathcal{H}^2_{\tilde{\omega}_\varepsilon,+}$.
 \item[$({\rm iii})$] $\tilde{\omega}^-_{\alpha,\varepsilon},\tilde{\Xi}_{p,i,\varepsilon}$ form a basis of $\mathcal{H}^2_{\tilde{\omega}_\varepsilon,-}$. Moreover, their $L^2$-inner products are given by 
\begin{align*}
\langle \tilde{\Xi}_{p,i,\varepsilon},\tilde{\Xi}_{q,j,\varepsilon}\rangle_{\tilde{g}_\varepsilon}&=\frac{\varepsilon^4\{ [c_1(A_{p,i})]_X\cdot Q\}
\{ [c_1(A_{q,j})]_X\cdot Q\}}{\vol(\T)}\quad(p\neq q),\\
\langle \tilde{\Xi}_{p,i,\varepsilon},\tilde{\Xi}_{p,j,\varepsilon}\rangle_{\tilde{g}_\varepsilon}&=L_{p,ij}+\frac{\varepsilon^4\{ [c_1(A_{p,i})]_X\cdot Q\}\{ [c_1(A_{p,j})]_X\cdot Q\}}{\vol(\T)},\\
\langle \tilde{\omega}^-_{\alpha,\varepsilon},\tilde{\Xi}_{p,i,\varepsilon}\rangle_{\tilde{g}_\varepsilon}&=0,\\
\langle \tilde{\omega}^-_{\alpha,\varepsilon},\tilde{\omega}^-_{\beta,\varepsilon}\rangle_{\tilde{g}_\varepsilon}&=\delta_{\alpha\beta}{\rm vol}(\T).
\end{align*}
 \item[$({\rm iv})$] As $\varepsilon\to 0$, we have the following convergence in the sense of Definition \ref{def conv bubble},
\begin{align*}
&\tilde{\omega}_\varepsilon\stackrel{\T\setminus S}{\rightarrow}\omega_0,\quad\varepsilon^{-2}\tilde{\omega}_\varepsilon\stackrel{S}{\rightarrow}(\eta_p)_{p\in S},\\
&\tilde{\omega}^-_{\alpha,\varepsilon}\stackrel{\T\setminus S}{\rightarrow}\omega^-_\alpha,\quad\varepsilon^{-2}\tilde{\omega}^-_{\alpha,\varepsilon}\stackrel{S}{\rightarrow}(\sqrt{-1}\ddb\nu_{p,\alpha})_{p\in S},\\
&\tilde{\Xi}_{p,i,\varepsilon}\stackrel{\T\setminus S}{\rightarrow}0,\quad\tilde{\Xi}_{p,i,\varepsilon}\stackrel{S}{\rightarrow}(c_1(A_{p,i}),(0)_{p'\in S\setminus\{ p\}}).
\end{align*}
\end{itemize}
\end{thm}

\subsection{Proof of Theorem \ref{thm main 7}}
We put $\Xi_\varepsilon=\Xi'_{\varepsilon,p,i}$ or $\omega^-_{\alpha,\varepsilon}$ and $F_\varepsilon:=\frac{\Xi_\varepsilon\wedge \tilde{\omega}_\varepsilon}{\tilde{\omega}_\varepsilon^2}$. Denote by $g_\varepsilon,\tilde{g}_\varepsilon$ the K\"ahler metric of $\omega_\varepsilon,\tilde{\omega}_\varepsilon$ and $\nabla,\tilde{\nabla}$ the Levi-Civita connection of $g_\varepsilon,\tilde{g}_\varepsilon$, respectively. 
We consider a pair $(\lambda_\varepsilon,G_\varepsilon)$ of a constant $\lambda_\varepsilon$ and a function $G_\varepsilon$ on $X$ satisfying 
\begin{align*}
\left( \lambda_\varepsilon\tilde{\omega}_\varepsilon+\Xi_\varepsilon+\sqrt{-1}\ddb G_\varepsilon\right)\wedge\tilde{\omega}_\varepsilon=0.
\end{align*}
Here, $\lambda_\varepsilon$ is determined by 
\begin{align*}
\lambda_\varepsilon
&=-\frac{\int_X\Xi_\varepsilon\wedge\tilde{\omega}_\varepsilon}{\int_X\tilde{\omega}_\varepsilon^2}=-\frac{\varepsilon^2[\Xi_\varepsilon]\cdot Q}{{\rm vol}(\T)\sqrt{1+\varepsilon^4Q^2/{\rm vol}(\T)}},
\end{align*}
by Proposition \ref{prop cup prod}, then there exists $G_\varepsilon$ satisfying the above equation. Here, $G_\varepsilon$ is determined uniquely up to additive constants. Note that 
\[ \lambda_\varepsilon=
\left\{
\begin{array}{cl}
0 & \mbox{ if }\Xi_\varepsilon=\omega^-_{\alpha,\varepsilon}, \\
-\frac{\varepsilon^4[c_1(A_{p,i})]_X\cdot Q}{{\rm vol}(\T)\sqrt{1+\varepsilon^4Q^2/{\rm vol}(\T)}}, & \mbox{ if } \Xi_\varepsilon=\Xi'_{p,i,\varepsilon}.
\end{array}
\right.
\]

By the construction of $\Xi_\varepsilon$, we have shown that $\Xi_\varepsilon\wedge \omega_\varepsilon=0$ on $X\setminus\bigcup_{p}A(p,\varepsilon^{1/2})$. Therefore, 
\begin{align*}
\Xi_\varepsilon\wedge \tilde{\omega}_\varepsilon&=\Xi_\varepsilon\wedge(\tilde{\omega}_\varepsilon-W_\varepsilon^{-1/2}\omega_\varepsilon)\mbox{ on }X\setminus\bigcup_{p}A(p,\varepsilon^{1/2}),\\
\Xi_\varepsilon\wedge \tilde{\omega}_\varepsilon&=\Xi_\varepsilon\wedge(\tilde{\omega}_\varepsilon-W_\varepsilon^{-1/2}\omega_\varepsilon)+W_\varepsilon^{-1/2}\Xi_\varepsilon\wedge\omega_\varepsilon\mbox{ on }A(p,\varepsilon^{1/2}).
\end{align*}

\begin{prop}
We have $\| F_\varepsilon\|_{C^{k,\alpha}_{\delta-2,\tilde{g}_\varepsilon}}\lesssim \varepsilon^{3-\delta/2}$. 
\label{prop holder F}
\end{prop}
\begin{proof}
By Propositions \ref{prop decay bubble 2 forms} and \ref{prop decay outside 2 forms}, we can see 
\begin{align*}
|\Xi_\varepsilon|_{g_\varepsilon}
&\lesssim 1,\\
|\nabla^k(\Xi_\varepsilon)|_{g_\varepsilon}
&\lesssim\varepsilon^4\sigma_\varepsilon^{-4-k}\quad (k\ge 1).
\end{align*}
Moreover, we have $\Xi_\varepsilon\wedge \omega_\varepsilon=0$ on $X\setminus\bigcup_{p}A(p,\varepsilon^{1/2})$ and 
\begin{align*}
\left|\nabla^k ( \Xi_\varepsilon\wedge W_\varepsilon^{-1/2}\omega_\varepsilon)\right|_{g_\varepsilon}&\lesssim \varepsilon^{2-k/2}\quad\mbox{on }A(p,\varepsilon^{1/2}).
\end{align*}
Since $\varepsilon^{1/2}\lesssim \sigma_\varepsilon$ and $\varepsilon^{-1/2}\lesssim\sigma_\varepsilon^{-1}$ on $A(p,\varepsilon^{1/2})$, we have 
\begin{align*}
\left|\nabla^k ( \Xi_\varepsilon\wedge W_\varepsilon^{-1/2}\omega_\varepsilon)\right|_{g_\varepsilon}&\lesssim \varepsilon^{3-\delta/2}\sigma_\varepsilon^{\delta-2-k}\quad\mbox{on }X.
\end{align*}
Recall that $\| W_\varepsilon^{1/2}\tilde{\omega}_\varepsilon-\omega_\varepsilon\|_{C^{k,\alpha}_{\delta-2,g_\varepsilon}}\lesssim \varepsilon^{3-\delta/2}$. Then we have $|\nabla^k(\tilde{\omega}_\varepsilon-W_\varepsilon^{-1/2}\omega_\varepsilon)|\lesssim\varepsilon^{3-\delta/2}\sigma_\varepsilon^{\delta-2-k}$ and 
\begin{align*}
\left|\nabla^k\{ \Xi_\varepsilon\wedge (\tilde{\omega}_\varepsilon-W_\varepsilon^{-1/2}\omega_\varepsilon)\}\right|_{g_\varepsilon}&\lesssim \varepsilon^{3-\delta/2}\sigma_\varepsilon^{\delta-2-k}.
\end{align*}
Therefore, we obtain 
\begin{align*}
\sigma_\varepsilon^{-(\delta-2)+k}\left|\nabla^k(\Xi_\varepsilon\wedge \tilde{\omega}_\varepsilon)\right|
\lesssim \varepsilon^{3-\delta/2},
\end{align*}
which implies $\| F_\varepsilon\|_{C^{k,\alpha}_{\delta-2,g_\varepsilon}}\lesssim\varepsilon^{3-\delta/2}$, then Proposition \ref{prop Ck comparison} gives $\| F_\varepsilon\|_{C^{k,\alpha}_{\delta-2,\tilde{g}_\varepsilon}}\lesssim\varepsilon^{3-\delta/2}$.
\end{proof}

\begin{prop}
Let $G_\varepsilon\in C^\infty(X)$ be the unique function satisfies 
$\frac{1}{2}L_{\tilde{\omega}_\varepsilon}(G_\varepsilon)=-F_\varepsilon-\lambda_\varepsilon$ and $\int_XG_\varepsilon\tilde{\omega}_\varepsilon^2=0$. Then we have $\| G_\varepsilon\|_{C^{k+2,\alpha}_{\delta,\tilde{g}_\varepsilon}}\lesssim\varepsilon^{3-\delta/2}$. 
\label{prop holder of potential}
\end{prop}
\begin{proof}
Since $\tilde{g}_\varepsilon$ satisfy \eqref{eq asymp1}\eqref{eq asymp2}, we can apply Proposition \ref{prop weight sch int 0}. Then there are constants $K_{k,\delta}$ such that  
\begin{align*}
\| G_\varepsilon\|_{C^{k+2,\alpha}_{\delta,\tilde{g}_\varepsilon}}
&\le K_{k,\delta}\| F_\varepsilon+\lambda_\varepsilon\|_{C^{k,\alpha}_{\delta-2,\tilde{g}_\varepsilon}}\lesssim\varepsilon^{3-\delta/2}.
\end{align*}
\end{proof}
Let $G_\varepsilon$ be the function obtained by Proposition \ref{prop holder of potential} and put 
\begin{align*}
\mathcal{A}(\Xi_\varepsilon):=\Xi_\varepsilon+\lambda_\varepsilon\tilde{\omega}_\varepsilon+\sqrt{-1}\ddb G_\varepsilon.
\end{align*}
By the construction, we have 
\begin{align}
\mathcal{A}(\Xi_\varepsilon)\in\Omega^{1,1}(X),\quad
\mathcal{A}(\Xi_\varepsilon)\wedge\tilde{\omega}_\varepsilon=0.\label{eq harm proj}
\end{align}
\begin{prop}
Let $\mathcal{A}(\Xi_\varepsilon)$ be as above. Then $\mathcal{A}(\Xi_\varepsilon)$ is an anti-self-dual and harmonic $2$-form on $(X,\tilde{g}_\varepsilon)$. 
\end{prop}
\begin{proof}
Now, $(X,\tilde{g}_\varepsilon)$ is a hyper-K\"ahler manifold of dimension $4$. Then $\tilde{\omega}_\varepsilon,{\rm Re} (\Omega),{\rm Im} (\Omega)$ form a basis of self-dual $2$-forms on every $p\in X$. Since $\mathcal{A}(\Xi_\varepsilon)\wedge\tilde{\omega}_\varepsilon=\mathcal{A}(\Xi_\varepsilon)\wedge \Omega=0$, $\mathcal{A}(\Xi_\varepsilon)$ is anti-self-dual. Moreover, $\mathcal{A}(\Xi_\varepsilon)$ is closed by the construction, which gives  $d*\mathcal{A}(\Xi_\varepsilon)=-d\mathcal{A}(\Xi_\varepsilon)_\varepsilon=0$, hence it is harmonic. 
\end{proof}
Now, we put $\Xi_\varepsilon^0,\Xi_\varepsilon^1=\Xi'_{p,i,\varepsilon}$ or $\omega^-_{\alpha,\varepsilon}$. For a Riemannian metric $g$, denote by $\langle\cdot,\cdot\rangle_g$ the $L^2$-inner product with respect to $g$, then we have  
\begin{align*}
\langle\mathcal{A}(\Xi_\varepsilon^0),\mathcal{A}(\Xi_\varepsilon^1)\rangle_{\tilde{g}_\varepsilon}=-\int_X\mathcal{A}(\Xi_\varepsilon^0)\wedge \mathcal{A}(\Xi_\varepsilon^1).
\end{align*}
Put $\lambda_\varepsilon^i=-\int_X\Xi_\varepsilon^i\wedge\tilde{\omega}_\varepsilon/\int_X\tilde{\omega}_\varepsilon^2$, then 
\begin{align*}
\langle\mathcal{A}(\Xi_\varepsilon^0),\mathcal{A}(\Xi_\varepsilon^1)\rangle_{\tilde{g}_\varepsilon}
&=-\int_X(\Xi_\varepsilon^0+\lambda_\varepsilon^0\tilde{\omega}_\varepsilon)\wedge(\Xi_\varepsilon^1+\lambda_\varepsilon^1\tilde{\omega}_\varepsilon)\\
&=-\int_X\Xi_\varepsilon^0\wedge\Xi_\varepsilon^1+\lambda_\varepsilon^0\lambda_\varepsilon^1{\rm vol}(\T).
\end{align*}
From now on, we write 
\begin{align*}
\tilde{\Xi}'_{p,i,\varepsilon} := \varepsilon^{-2}\mathcal{A}(\Xi'_{p,i,\varepsilon}),\quad
\tilde{\omega}^-_{\alpha,\varepsilon} :=\mathcal{A}(\omega^-_{\alpha,\varepsilon}).
\end{align*}
By \eqref{eq RFK class} and the computation of $\lambda_\varepsilon$, we have 
\begin{align*}
[\tilde{\omega}^-_{\alpha,\varepsilon}] 
&=[\omega^-_{\alpha,\varepsilon}]=[\omega^-_\alpha]_X,\\
[\tilde{\Xi}'_{p,i,\varepsilon}] 
&= \varepsilon^{-2}[\Xi'_{p,i,\varepsilon}]-
\frac{\varepsilon^2[c_1(A_{p,i,\varepsilon})]_X\cdot Q}{{\rm vol}(\T)\sqrt{1+\varepsilon^4Q^2/{\rm vol}(\T)}}[\tilde{\omega}_\varepsilon]\\
&= [c_1(A_{p,i})]_X-\frac{\varepsilon^2[c_1(A_{p,i,\varepsilon})]_X\cdot Q}{{\rm vol}(\T)+\varepsilon^4Q^2}([\omega_0]_X+\varepsilon^2 Q).
\end{align*}
Moreover, we define anti-self-dual $2$-forms $\tilde{\Xi}_{p,i,\varepsilon}$ by 
\begin{align*}
\tilde{\Xi}_{p,i,\varepsilon}
&:= \tilde{\Xi}'_{p,i,\varepsilon}
+\frac{\varepsilon^4[c_1(A_{p,i})]_X\cdot Q}{\vol(\T)} \sum_{q\in S}\sum_{j=1}^{N_{\Gamma_q}}\eta_q^j\tilde{\Xi}'_{q,j,\varepsilon}.
\end{align*}
By direct calculation, we have 
\begin{align*}
[\tilde{\Xi}_{p,i,\varepsilon}]
&=  [c_1(A_{p,i})]_X+\frac{\varepsilon^2\sum_jL_{p,ij}\eta_p^j}{{\rm vol}(\T)}[\omega_0]_X.
\end{align*}
In particular, $\{\tilde{\Xi}_{p,i,\varepsilon}\}_{p,i}$ is linearly independent.

For anti-self-dual $2$-forms $a,b\in\Omega^2(X)$, we can see $\langle a,b\rangle_{\tilde{g}_\varepsilon}=\int_Xa\wedge *b=-\int_Xa\wedge b$. By Proposition \ref{prop cup prod}, we obtain 
\begin{align*}
\langle \tilde{\Xi}_{p,i,\varepsilon},\tilde{\Xi}_{q,j,\varepsilon}\rangle_{\tilde{g}_\varepsilon}&=\frac{\varepsilon^4\{ [c_1(A_{p,i})]_X\cdot Q\}
\{ [c_1(A_{q,j})]_X\cdot Q\}}{\vol(\T)}\quad(p\neq q),\\
\langle \tilde{\Xi}_{p,i,\varepsilon},\tilde{\Xi}_{p,j,\varepsilon}\rangle_{\tilde{g}_\varepsilon}&=L_{p,ij}+\frac{\varepsilon^4\{ [c_1(A_{p,i})]_X\cdot Q\}\{ [c_1(A_{p,j})]_X\cdot Q\}}{\vol(\T)},\\
\langle \tilde{\omega}^-_{\alpha,\varepsilon},\tilde{\Xi}_{p,i,\varepsilon}\rangle_{\tilde{g}_\varepsilon}&=0,\\
\langle \tilde{\omega}^-_{\alpha,\varepsilon},\tilde{\omega}^-_{\beta,\varepsilon}\rangle_{\tilde{g}_\varepsilon}&=\delta_{\alpha\beta}{\rm vol}(\T).
\end{align*}

\begin{prop}
We have the following convergence.
\begin{itemize}
\setlength{\parskip}{0cm}
\setlength{\itemsep}{0cm}
 \item[$({\rm i})$] $\tilde{\omega}^-_{\alpha,\varepsilon}\stackrel{\T\setminus S}{\rightarrow} \omega^-_\alpha$ and $\varepsilon^{-2}\tilde{\omega}^-_{\alpha,\varepsilon}\stackrel{S}{\rightarrow} (\sqrt{-1}\ddb\nu_{p,\alpha})_{p\in S}$. 
 \item[$({\rm ii})$] $\tilde{\Xi}_{p,i,\varepsilon}\stackrel{\T\setminus S}{\rightarrow} 0$ and $\tilde{\Xi}_{p,i,\varepsilon}\stackrel{S}{\rightarrow} (c_1(A_{p,i}),(0)_{p'\in S\setminus \{ p\}})$. 
\end{itemize}
\end{prop}
\begin{proof}
By Proposition, \ref{prop holder of potential} we have $\|\tilde{\omega}^-_{\alpha,\varepsilon}-\omega^-_{\alpha,\varepsilon}\|_{C^{k,\alpha}_{\delta-2,g_\varepsilon}}\lesssim \varepsilon^{3-\delta/2}$. 
By the construction of $\omega^-_{\alpha,\varepsilon}$ and by Lemma \ref{lem conv 2-form}, we can show $({\rm i})$. Similarly, we have 
$\|\tilde{\Xi}'_{p,i,\varepsilon}-\varepsilon^{-2}\Xi'_{p,i,\varepsilon}\|_{C^{k,\alpha}_{\delta-2,g_\varepsilon}}\lesssim \varepsilon^{1-\delta/2}$ and $\|\tilde{\Xi}_{p,i,\varepsilon}-\tilde{\Xi}'_{p,i,\varepsilon}\|_{C^{k,\alpha}_{\delta-2,g_\varepsilon}}\lesssim \varepsilon^4\|\tilde{\Xi}'_{p,i,\varepsilon}\|_{C^{k,\alpha}_{\delta-2,g_\varepsilon}}$. Moreover, by Proposition \ref{prop decay bubble 2 forms}, we have $\|\Xi'_{p,i,\varepsilon}\|_{C^{k,\alpha}_{\delta-2,g_\varepsilon}}\lesssim\varepsilon^{2-\delta}$ by $-2<\delta<0$. Therefore, $\|\tilde{\Xi}_{p,i,\varepsilon}-\varepsilon^{-2}\Xi'_{p,i,\varepsilon}\|_{C^{k,\alpha}_{\delta-2,g_\varepsilon}}\lesssim \varepsilon^{1-\delta/2}$. Since $1-\delta/2>0$, we have $({\rm ii})$ by Lemma \ref{lem conv 2-form}.
\end{proof}

\section{Examples}\label{sec ex}
Here, we give some examples of $\Lambda\subset \C^2$ and $\Gamma\subset SU(2)$ with $\Gamma\cdot\Lambda\subset\Lambda$. We write 
\[ \C^2=\left\{\left.
 \left(
\begin{array}{c}
z \\
w
\end{array}
\right )\right|\, z,w\in\C\right\}
\]
and for the natural quotient map $\pi_{\Lambda,\Gamma}\colon\C^2\to \T=\T^4_\Lambda/\Gamma$, we write 
\[
\pi_{\Lambda,\Gamma}
 \left(
\begin{array}{c}
z \\
w
\end{array}
\right )= \left[
\begin{array}{c}
z \\
w
\end{array}
\right].
\]

$(1)$ Let $\Lambda$ be any lattice of rank $4$ and $\Gamma=\{ \pm 1\}\cong\Z_2$. Then $\Gamma\cdot\Lambda\subset\Lambda$. Moreover, $\# S=16$ and $\Gamma_p=\Z_2$ for all $p\in S$. In this case, $d_\Gamma=3$ and $N_{\Gamma_p}=1$, hence $d_\Gamma+\sum_{p\in S}N_{\Gamma_p}=19=b^2_-(K3)$.

$(2)$ Put $\Gamma=\Z_4$. In this case, we have $d_{\Gamma}=1$. Here, we need to choose $\Lambda$ carefully so that $\Gamma\cdot\Lambda\subset\Lambda$ holds. We suppose 
\[ \Z_4=
\left\{ \left. \left(
\begin{array}{ccc}
\sqrt{-1} & 0 \\
0 & -\sqrt{-1} 
\end{array}
\right )^m\right|\, m=0,1,2,3\right\},
\]
\[ \Lambda=
\Z\left (
\begin{array}{c}
1 \\
0
\end{array}
\right )\oplus\Z\left (
\begin{array}{c}
\sqrt{-1}  \\
0
\end{array}
\right )\oplus
\Z\left (
\begin{array}{c}
0 \\
1
\end{array}
\right )\oplus
\Z\left (
\begin{array}{c}
0 \\
\sqrt{-1} 
\end{array}
\right ).
\]
Then we have $\Gamma\cdot\Lambda\subset\Lambda$. The points $p\in S$ with $\Gamma_p=\Z_4$ are 
\[ 
\left[
\begin{array}{c}
0 \\
0
\end{array}
\right],\,\left[
\begin{array}{c}
\frac{1+\sqrt{-1}}{2} \\
0
\end{array}
\right],\, \left[
\begin{array}{c}
0 \\
\frac{1+\sqrt{-1}}{2} 
\end{array}
\right],\, \left[
\begin{array}{c}
\frac{1+\sqrt{-1}}{2} \\
\frac{1+\sqrt{-1}}{2} 
\end{array}
\right].
\]
The points $p\in S$ with $\Gamma_p=\Z_2$ are 
\[ 
\left[
\begin{array}{c}
1/2 \\
0
\end{array}
\right],\,\left[
\begin{array}{c}
1/2 \\
\frac{1+\sqrt{-1}}{2} 
\end{array}
\right],\,\left[
\begin{array}{c}
0 \\
\frac{1}{2}
\end{array}
\right],\, \left[
\begin{array}{c}
\frac{1+\sqrt{-1}}{2} \\
1/2
\end{array}
\right],\, \left[
\begin{array}{c}
1/2 \\
1/2
\end{array}
\right],\, \left[
\begin{array}{c}
1/2 \\
\sqrt{-1}/2
\end{array}
\right].
\]
We can see that $N_{\Z_k}=k-1$. Therefore, there are $4$ points with $N_{\Gamma_p}=3$ in $S$ and $6$ points with $N_{\Gamma_p}=1$ in $S$. Consequently, we have 
\begin{align*}
d_\Gamma+\sum_{p\in S}N_{\Gamma_p}
=1+3\times 4+ 1\times 6=19.
\end{align*}

$(3)$ Next, put 
\[ \Gamma=
\left\{ \left. \left(
\begin{array}{ccc}
\sqrt{-1} & 0 \\
0 & -\sqrt{-1} 
\end{array}
\right )^m\left(
\begin{array}{ccc}
0 & -1 \\
1 & 0 
\end{array}
\right )^l\right|\, m=0,1,2,3,\, l=0,1\right\},
\]
and let $\Lambda$ be as in the previous example. Then $\Gamma\cdot\Lambda\subset\Lambda$. In this case, $\Gamma$ is not abelian, hence $d_\Gamma=0$. The points $p\in S$ with $\Gamma_p=\Gamma$ are 
\[ 
\left[
\begin{array}{c}
0 \\
0
\end{array}
\right],\, \left[
\begin{array}{c}
\frac{1+\sqrt{-1}}{2} \\
\frac{1+\sqrt{-1}}{2} 
\end{array}
\right].
\]
The points $p\in S$ with $\Gamma_p\cong\Z_4$ are 
\[ 
\left[
\begin{array}{c}
\frac{1+\sqrt{-1}}{2} \\
0
\end{array}
\right],\,\left[
\begin{array}{c}
1/2 \\
1/2
\end{array}
\right],\, \left[
\begin{array}{c}
1/2 \\
\sqrt{-1}/2
\end{array}
\right],
\]
the points $p\in S$ with $\Gamma_p\cong\Z_2$ are 
\[ 
\left[
\begin{array}{c}
1/2 \\
0
\end{array}
\right],\,\left[
\begin{array}{c}
1/2 \\
\frac{1+\sqrt{-1}}{2}
\end{array}
\right].
\]
By the McKay correspondence, the corresponding Dynkin diagram of $\Gamma$ is of type $D_4$, which has $4$ vertices. 
Therefore, we have $N_\Gamma=4$ and 
\begin{align*}
d_\Gamma+\sum_{p\in S}N_{\Gamma_p}
=0+4\times 2+ 3\times 3+1\times 2=19.
\end{align*}

\bibliographystyle{plain}

\end{document}